\documentclass[final,onefignum,onetabnum]{siamart190516}

\usepackage{amsfonts}
\usepackage{graphicx}
\usepackage{epstopdf}
\ifpdf
  \DeclareGraphicsExtensions{.eps,.pdf,.png,.jpg}
\else
  \DeclareGraphicsExtensions{.eps}
\fi

\usepackage[round]{natbib} 
\usepackage{mathtools}
\usepackage{enumitem}
\usepackage{blkarray}
\usepackage{undertilde_closer}

\newsiamremark{remark}{Remark}
\newsiamremark{assumption}{Assumption}
\crefname{assumption}{Assumption}{Assumption} 

\title{A New Truncation Algorithm for Markov Chain Equilibrium Distributions with Computable Error Bounds}

\author{Alex Infanger\thanks{Institute for Computational \& Mathematical Engineering, Stanford University, Stanford, CA 94305.
(\email{alexinf@stanford.edu}, \url{https://stanford.edu/\~alexinf/}).}
\and Peter W. Glynn\footnotemark[2] \thanks{Department of Management Science \& Engineering, Stanford University, Stanford, CA 94305.
  (\email{glynn@stanford.edu}, \url{https://stanford.edu/\~glynn/}). }
  }

\usepackage{amsopn}

\newcommand{\R}{\mathbb{R}}
\newcommand\norm[1]{\left\lVert#1\right\rVert}
\newcommand{\mnorm}[1]{{\left\vert\kern-0.25ex\left\vert\kern-0.25ex\left\vert #1 
    \right\vert\kern-0.25ex\right\vert\kern-0.25ex\right\vert}}

\usepackage{xurl}
\usepackage{subfig}
\usepackage[normalem]{ulem}

\usepackage{comment}
\includecomment{includefigures}

\ifpdf
\hypersetup{
  pdftitle={A New Truncation Algorithm for Markov Chain Equilibrium Distributions with Computable Error Bounds},
  pdfauthor={A. Infanger, P. W. Glynn}
}
\fi


\begin{document}

\maketitle

\begin{abstract}
This paper introduces a new algorithm for numerically computing equilibrium (i.e. stationary) distributions for Markov chains and Markov jump processes with either a very large finite state space or a countably infinite state space. The algorithm is based on a ratio representation for equilibrium expectations in which the numerator and denominator correspond to expectations defined over paths that start and end within a given return set $K$. When $K$ is a singleton, this representation is a well-known consequence of regenerative process theory. For computational tractability, we ignore contributions to the path expectations corresponding to excursions out of a given truncation set $A$. This yields a truncation algorithm that is provably convergent as $A$ gets large. Furthermore, in the presence of a suitable Lyapunov function, we can bound the path expectations, thereby providing computable and convergent error bounds for our numerical procedure. Our paper also provides a computational comparison with two other truncation methods that come with computable error bounds. The results are in alignment with the observation that our bounds have associated computational complexities that typically scale better as the truncation set gets bigger.
\end{abstract}


\begin{AMS}
  {60J10, 60J22, 65C40, 60J27, 60J45, 60J74, 90C05, 90C90} 
\end{AMS}

\section{Introduction}

Let $X=(X_n:n\geq 0)$ be an irreducible positive recurrent Markov chain on a discrete state space $\mathcal{S}$ with a very large or infinite number of states, having a one-step transition matrix $P=(P(x,y): x,y\in\mathcal{S})$. Such a Markov chain has a unique \emph{equilibrium distribution} $\pi=(\pi(x):x\in\mathcal{S})$ that we will encode as a row vector. Such distributions are also equivalently known as steady-state distributions, stationary distributions, and invariant distributions for the Markov chain $X$. When the cardinality $|S|$ is very large or infinite, one can not numerically compute the equilibrium distribution exactly via linear algebraic methods, and state space ``truncation'' to some finite subset $A\subseteq S$ then becomes necessary. 

There is a significant literature on the equilibrium distribution truncation problem, starting with early work of \citet{senetaFiniteApproximationsInfinite1967}. Many of these papers take the 
sub-stochastic truncated transition matrix, stochasticize the rows of this matrix, and use 
the equilibrium distribution of the stochasticized matrix as an approximation to the
equilibrium distribution of the original transition matrix. These methods, that involve addition of probability mass to each row of the truncated matrix, are known as “truncation/augmentation” algorithms. \citet{wolfApproximationInvariantProbability1980} used Foster's criterion to prove convergence for certain families of augmentations, and \citet{senetaComputingStationaryDistribution1980} proved that tightness of the approximating equilibrium distributions suffices to guarantee convergence. \citet{gibsonAugmentedTruncationsInfinite1987} established convergence for general augmentations when the original transition matrix is upper Hessenberg or contains a column uniformly bounded away from zero, and \citet{gibsonMonotoneInfiniteStochastic1987} established convergence for general augmentations for stochastically monotone chains. \citet{heymanApproximatingStationaryDistribution1991} used ideas from the theory of regeneration to give a probabilistic condition for verifying convergence, while \citet{zhaoCensoredMarkovChain1996} observed that an optimal augmentation is achieved by the ``censored Markov chain''; see also \cref{sec::Approximation}. Recently, \citet{infanger2022convergence} provided a Lyapunov condition that guarantees convergence of general augmentations. \citet{infangerConvergenceTruncationScheme2022} give conditions under which ``fixed state'' augmentations converge in both discrete and continuous state spaces.

\citet{tweedieTruncationApproximationsInvariant1998} developed convergence results for geometrically ergodic chains (as well as other chains) and provided error bounds for the approximation. Recent works finding error bounds under various conditions (e.g. geometric drift, block-monotonicity) include \citet{herveApproximatingMarkovChains2014}, \citet{masuyamaErrorBoundsAugmented2015}, and \citet{masuyamaErrorBoundsAugmented2016}. Analogous assumptions can lead to error bounds for the corresponding truncation problem in continuous time. See, for example, \citet{guptaFiniteStateProjection2017} (geometric drift) and \citet{masuyamaContinuoustimeBlockmonotoneMarkov2017} (block-monotonicity). \citet{liuErrorBoundsAugmented2018} provide an error bound for general augmentations in continuous time under a Lyapunov condition. An approach to producing an approximation and error bounds based on linear progamming is presented in \citet{kuntzBoundingStationaryDistributions2019}. Current applications related to bio-chemical reaction networks have stimulated significant recent interest in the truncation problem, since the associated chains tend to have infinite state spaces. The survey of \citet{kuntzStationaryDistributionsContinuoustime2021} discusses the current state-of-the-art.

In this paper, we introduce a new method for computing approximations to $\pi$ based on such truncations, which comes with a computable a posteriori theoretically valid error bound. Our method does not use truncation/augmentation methods. Rather, our method exploits a ratio representation for the equilibrium distribution in which the numerator and the denominator are expectations taken over paths describing excursions starting and finishing in some chosen ``return set'' $K$; see \cref{eq::GeneralizedRegenerativeRepresentation}. 

In particular, we do not stochasticize the ``northwest corner'' of the one-step transition matrix as do standard augmentation algorithms. Instead, our approach approximates the excursion expectations associated with our ratio representation; see \cref{sec::Approximation}. The approximation is guaranteed to converge as the size of the truncation set $A$  gets large; see \cref{prop_convergence}. Furthermore, the excursion expectations that are central to our approach are well suited to the construction of rigorous bounds, in the presence of a suitable Lyapunov function: see \cref{assumption::Lyapunov1} and \cref{prop-Lyapunov}. Our approach therefore leads to upper and lower bounds on the excursion expectations that are easily computable, and that appear to yield good bounds for equilibrium quantities in our computational examples.

The bounds we develop do not assume knowledge of rates of convergence of the $n$-step transition probabilities to the equilibrium limit, as do many of the existing bounds; see, for example, \citet{masuyamaErrorBoundsAugmented2015,herveApproximatingMarkovChains2014,tweedieTruncationApproximationsInvariant1998}. Such bounds tend to be (very) conservative, and are difficult to compute in many examples. By contrast, our error bounds depend only on the knowledge of a stochastic Lyapunov function $g$; see \cref{assumption::Lyapunov1}. Use of Lyapunov functions to obtain error bounds is standard in the truncation literature. However, as noted earlier, our use of Lyapunov functions differs, in that it is used directly to bound the relevent excursion expectations. While such stochastic Lyapunov functions $g$ are themselves somewhat difficult to find in many examples, any rigorous error bound will need information related to the transition dynamics of $X$ from $A^c$ back into $A$; see \cref{sec::role-of-Lyapunov}. Such a Lyapunov bound is one vehicle for capturing the necessary information. 

The key contributions of this paper include:

\begin{enumerate}[label=\arabic*)]
\item  The introduction of our new algorithm for approximating $\pi$ and the proof of convergence (\cref{sec::Approximation});
\item  The development of new rigorous computable error bounds for equilibrium expectations, as well as a proof that the bounds converge (\cref{sec::EqExpBounds});
\item  The development of a computable bound on the weighted total variation distance of our approximation from the true distribution $\pi$ (\cref{sec::TV-bounds});
\item  Describing the connection of our row-normalized approximation to a conditioned chain (\cref{sec::probabilistic-perspective-on-row-normalization});
\item  Describing the connection between a special case of our approximation to the ``exit approximation'' discussed in \citet{senetaNonnegativeMatricesMarkov2006} in the context of ``fixed state'' augmentation (\cref{sec::exit_approx});
\item  Generalizing our theory to truncation approximations for Markov jump processes (\cref{sec::MarkovJumpProcess});
\item Providing a discussion of our algorithm when applied to reducible chains (\cref{sec::Reducible}).
\end{enumerate}
 Our paper also includes a computational section (\cref{sec::NumericalExperiments}) that discusses the implementation and performance of our algorithm and bounds on a number of different examples, and makes comparisons to an existing set of algorithms described by \citet{kuntzStationaryDistributionsContinuoustime2021}. A key element in our contribution is that our error bounds go to zero rapidly, even when our choice of Lyapunov function is poor; see \cref{sec::role-of-Lyapunov} for a description of this point.

\section{A Description of the Truncation Approximation to the Equilibrium Distribution\label{sec::Approximation}}

Our approach involves first selecting a finite subset $A\subseteq S$ that represents the truncation set. The set $A$ should be chosen as large as possible, given the computational constraints associated with solving the linear systems of equations that our algorithm generates. Within the truncation set $A$, we choose a ``return set'' $K\subseteq A$ that will play a major role in our approach. We will return to the choice of $K$ later in this paper.

To describe our truncation approximation to the equilibrium distribution $\pi=(\pi(x): x\in S)$ associated with $P$, we need some additional notation. For $x\in S$, let $P_x(\cdot)=P(\cdot|X_0=x)$ be the probability on the path-space of $X$, conditional on $X_0=x$, and let $E_x(\cdot)$ be its associated expectation operator. For a given distribution $\mu=(\mu(x):x\in S)$ (encoded as a row vector), let
\begin{align*}
P_\mu(\cdot) \overset{\Delta}{=} \sum_{x}^{}\mu(x)P_x(\cdot)
\end{align*}
and
\begin{align*}
E_\mu(\cdot)\overset{\Delta}{=}\sum_{x}^{}\mu(x)E_x(\cdot).
\end{align*}
Given the return set set $K$ and truncation set $A$, let
\begin{align*}
T_K &= \inf\{n\geq 1: X_n\in K\},\\
T &= \inf\{n\geq 1: X_n\in A^c\},
\end{align*}
so that $T_K$ is the ``return time'' to $K$, and $T$ is the exit time from $A$. We shall make extensive theoretical use of the $K$-valued Markov chain \mbox{$Y=(Y_n:n\geq 0)$} that has transition probabilities
\begin{align*}
P_K(x,y) = P_x(X_{T_K}=y)
\end{align*}
for $x,y\in S$. The Markov chain $Y$ is called the \emph{process on K}; see p.29 of \citet{oreyLectureNotesLimit1971}. It is also known as the \emph{censored Markov chain}. The irreducibility of $X$ implies that $Y$ is irreducible, so that $Y$ has a unique equilibrium distribution $\pi_K=(\pi_K(x):x\in S)$, where $\pi_K$ satisfies 
\begin{align*}
\pi_K = \pi_KP_K,
\end{align*}
and $P_K = (P_K(x,y):x,y\in K)$. \\

We rely on the following ``ratio'' expression for $\pi(\cdot)$, namely
\begin{align}
\pi(y) = \frac{E_{\pi_K}\sum_{j=0}^{T_K-1} I(X_j=y)}{E_{\pi_K}T_K}; \label{eq::GeneralizedRegenerativeRepresentation}
\end{align}
see p.32 of \citet{oreyLectureNotesLimit1971}. (The expression there sums from $j=1$ to $T_K$, but $X_{T_K}$ has the same distribution as $X_0$.) We note that when $K$ is a singleton (i.e. $K=\{z\}$ for some $z$), \cref{eq::GeneralizedRegenerativeRepresentation} reduces to the standard ratio formula that is a well-known consequence of regenerative process theory; see, for example, p.14 of \citet{asmussenAppliedProbabilityQueues2008}. The equality \cref{eq::GeneralizedRegenerativeRepresentation} expresses the equilibrium probabilities in terms of expectations taken over ``$K$-cycles''. By applying Fubini's theorem, \cref{eq::GeneralizedRegenerativeRepresentation} immediately implies that when we encode a function $r:S\rightarrow\mathbb{R}$ as a column vector,
\begin{align}
\pi r &=\frac{E_{\pi_K}\sum_{j=0}^{T_K-1}r(X_j)}{E_{\pi_K}T_K} \label{eq::RegenerativeRepresentation-pir}
\end{align}
for any non-negative $r$. Hence, \cref{eq::RegenerativeRepresentation-pir} provides a representation of the equilibrium expectation $\pi r$ in terms of a $K$-cycle, for general non-negative ``reward'' functions on $S$. (To handle functions of mixed sign, we can split the function into its positive and negative parts.)

Let $e$ be the function for which $e(x)=1$ for $x\in S$. Then \cref{eq::RegenerativeRepresentation-pir} asserts that
\begin{align*}
\pi r = \frac{\kappa(r)}{\kappa(e)},
\end{align*}
where 
\begin{align*}
\kappa(w) = E_{\pi_K}\sum_{j=0}^{T_K-1}w(X_j)
\end{align*}
for a generic non-negative function $w$.

We now provide a representation of $\kappa(w)$ that will form the theoretical basis for our algorithm. To describe our result, we first represent our transition matrix $P$ in block-partitioned form, where the blocks are represented by $K,A',$ and $A^c$, with $A'\overset{\Delta}{=}A-K$. In particular, we write
\begin{align*}
P=
\begin{blockarray}{cccc}
& K & A' & A^c \\
\begin{block}{c(ccc)}
K & P_{11} & P_{12} & P_{13} \\
A' & P_{21} & P_{22} & P_{23} \\
A^c & P_{31} & P_{32} & P_{33} \\
\end{block}
\end{blockarray}\ \ \ .
\end{align*}
Similarly, we write the column vector $w$ in block-partitioned form, namely as $w^T=(w_1^T,w_2^T,w_3^T)$. Furthermore, we set 
\begin{align*}
B = \begin{pmatrix}
P_{22} & P_{23}\\
P_{32} & P_{33}
\end{pmatrix},
\end{align*}
and $\bar w^T = (w_2^T, w_3^T)$. Finally, we write $\sum_{n=0}^{\infty} B^n\bar w$ in block-partitioned form over the blocks corresponding to $A'$ and $A^c$, namely
\begin{align*}
\sum_{n=0}^{\infty}B^n\bar w \overset{\Delta}{=}\begin{pmatrix}
\eta'(\bar w)\\
\eta(\bar w)
\end{pmatrix}.
\end{align*}
We can now state the theorem that provides the basis for our algorithm.

\begin{theorem}\label{thm::linear_algebra_representation_for_sum_of_rewards} If $w$ is non-negative, then
\begin{align}
\kappa(w) = \pi_K\left[w_1 + P_{12}(I-P_{22})^{-1}w_2 + P_{13}\eta(\bar w) + P_{12}(I-P_{22})^{-1}P_{23}\eta(\bar w)\right].\label{eq::block_linear_system}
\end{align}
\end{theorem}

\begin{proof} For $x\in K$, we can write
\begin{align*}
E_x \sum_{j=0}^{T_K-1} w(X_j) &= w(x) + E_x\left[\sum_{j=1}^{(T_K\land T)-1}w(X_j)I(X_1\not\in K)\right]\ + \\
&\qquad\qquad\qquad\qquad\qquad\qquad\qquad  E_x\left[\sum_{j=T}^{T_K-1} w(X_j)I(T<T_K)\right],
\end{align*}
where $a\land b=\min(a,b)$ for $a,b\in \R$. The second term on the right-hand side is given by
\begin{align}
\sum_{y\in A'}^{}P(x,y)&\left[w(y)+ \sum_{y_1\in A'}^{}P(y,y_1)w(y_1) + \sum_{n=2}^{\infty} \sum_{\substack{y_1,...,y_n\\y_i\in A'}}^{}P(y,y_1)...P(y_{n-1},y_n)w(y_n)\right]\nonumber\\
&= (P_{12}w_2)(x) + \sum_{n=1}^{\infty}\left(P_{12}P_{22}^nw_2\right)(x)\nonumber\\
&=\sum_{n=0}^{\infty}\left(P_{12}P_{22}^nw_2\right)(x).\label{eq::sum_w_until_Tk_min_T_linear_system}
\end{align}
As for the third term, it equals
\begin{align}
E_x&\sum_{j=T}^{T_K-1} w(X_j)I(T=1) + E_x\sum_{j=T}^{T_K-1}w(X_j)I(1<T<T_k)\nonumber\\
&= \sum_{y\in A^c}^{}P(x,y)E_y\sum_{j=0}^{T_K-1}w(X_j) \ \ +\sum_{y\in A^c}^{}P_x(X_T=y, 1<T<T_K) E_y\sum_{j=0}^{T_K-1}w(X_j).\label{eq::rest_of_sum_of_w_linear_algebra}
\end{align}
Note that for $z\in K^c$, 
\begin{align*}
E_z\sum_{j=0}^{T_K-1}w(X_j) &= w(z) + \sum_{y\in K^c}^{}P(z,y)w(y) + \sum_{n=2}^{\infty}\sum_{\substack{y_1,...,y_n\\ y_i\not\in K}}^{}P(z,y_1)...P(y_{n-1},y_n)w(y_n)\\
&=\sum_{n=0}^{\infty}(B^nw)(z).
\end{align*}
On the other hand, for $x\in K$ and $y\in A^c$
\begin{align*}
P_x(X_T=y, 1<T<T_K) &= \sum_{z\in A'}^{}P(x,z)\biggr[P(z,y) + \sum_{z_1\in A'}^{}P(z,z_1)P(z_1,y)\ + \\
&\qquad\qquad\qquad \sum_{n=2}^{\infty}\sum_{\substack{z_1,...,z_n\\ z_i\in A'}}^{}P(z,z_1)...P(z_{n-1},z_n)P(z_n,y)\biggr]\\
&= \sum_{n=0}^{\infty}(P_{12}P_{22}^nP_{23})(x,y).
\end{align*}
It follows that
\begin{align}
\sum_{y\in A^c}^{}P(x,y)E_y \sum_{n=0}^{T_K-1}w(X_j) = \left(P_{13}\eta(\bar w)\right)(x)\label{eq::jump_into_Ac_and_sum_until_return_to_K_linear_algebra}
\end{align}
and
\begin{align}
\sum_{y\in A^c}^{}P_x(X_T=y, 1<T<T_K)E_y\sum_{n=0}^{T_K-1}w(X_j) = \sum_{n=0}^{\infty}(P_{12}P_{22}^n P_{23}\eta(\bar w))(x).\label{eq::jump_into_Aprime_and_sum_until_return_to_K_linear_algebra}
\end{align}
Finally, we note that for $x\in A'$,
\begin{align*}
\sum_{y\in A'}^{}P_{22}^n(x,y) = P_x(T\land T_K>n)\rightarrow 0
\end{align*}
as $n\rightarrow\infty$, so $P_{22}^n\rightarrow 0$ as $n\rightarrow\infty$ (by irreducibility and recurrence). Since $|A'|<\infty$, it follows that $(I-P_{22})^{-1}$ exists and equals $\sum_{n=0}^\infty P_{22}^n$ (see, for example, Lemma B.1 in \citet{senetaNonnegativeMatricesMarkov2006}). Consequently, \cref{eq::sum_w_until_Tk_min_T_linear_system} through \cref{eq::jump_into_Aprime_and_sum_until_return_to_K_linear_algebra} then yield the theorem.
\end{proof}

\begin{remark} We note that for $x\in K^c$, 
\begin{align*}
\sum_{y\in K^c}^{}B^n(x,y) = P_x(T_K>n)\rightarrow 0
\end{align*}
as $n\rightarrow\infty$ (by irreducibility and recurrence). However, since $B$ can be an infinite state matrix, we avoid writing $\sum_{n=0}^{\infty}B^n= (I-B)^{-1}$, to avoid technical questions related to the sense in which $I-B$ is non-singular. As we shall see, our approximation and bounds do not require any exploration of these non-singularity issues for $I-B$.

With \cref{thm::linear_algebra_representation_for_sum_of_rewards} in hand, note that when $K=\{z\}$, a singleton, then $\pi_K(z)=1$ and we can approximate $\pi r$ via
\begin{align}
\tilde \pi(r) \overset{\Delta}{=} \frac{r_1(z)+(P_{12}(I-P_{22})^{-1}r_2)(z)}{1+(P_{12}(I-P_{22})^{-1} e_2)(z)}=\frac{E_z\sum_{j=0}^{(T\land T_K)-1}r(X_j)}{E_zT\land T_K}.\label{eq::tilde-pir-singleton-case}
\end{align}
On the other hand, if $K$ is not a singleton, we need to approximate $\pi_K$. For $x,y\in K$, we note that 
\begin{align*}
P_K(x,y) &\overset{\Delta}{=} P_x(X_{T_K}=y, T_K<T) + P_x(X_{T_k}=y, T<T_K).
\end{align*}
We shall denote the first term on the right-hand side by $G(x,y)$. Note that
\begin{align*}
G(x,y)&=P(x,y) + \sum_{z\in A'}^{}P(x,z)P(z,y) + \sum_{z\in A',z_1\in A'}^{}P(x,z)P(z,z_1)P(z_1,y) \ + \\
&\qquad\qquad \sum_{n=2}^{\infty}\sum_{\substack{z, z_1,...,z_n\\ z\in A', z_i\in A'}}^{}P(x,z)P(z,z_1)... P(z_{n-1},z_n)P(z_n,y)\\
&= P(x,y) + \sum_{n=0}^{\infty}(P_{12}P_{22}^nP_{21})(x,y)
\end{align*}
so that
\begin{align}
G = P_{11} + P_{12}(I-P_{22})^{-1}P_{21}.\label{eq::computation-of-G-linear-algebra}
\end{align}
Clearly, $G$ is not stochastic, but is a good approximation to $P_K$ when $A$ is large. The following assumption will be in force in the remainder of this paper.  
\begin{assumption} Assume that the sets $K$ and $A$ have been chosen so that $G$ is irreducible. \label{assumption::G-irreducible}
\end{assumption}

In the presence of \cref{assumption::G-irreducible}, we now ``stochasticize'' $G$ in such a way as to preserve its irreducibility to obtain $\tilde G$; we will suggest two such methods shortly. With the irreducible stochasticized $\tilde G$, let $\tilde \pi_K$ be its associated (unique) equilibrium distribution. We then obtain the approximation
\begin{align}
\tilde \pi(r) = \frac{\tilde \pi_K\left[r_1 + P_{12}(I-P_{22})^{-1}r_2\right]}{\tilde \pi_K\left[e_1 + P_{12}(I-P_{22})^{-1}e_2\right]}=\frac{E_{\tilde \pi_K}\sum_{j=0}^{(T\land T_K)-1}r(X_j)}{E_{\tilde \pi_K}T\land T_K}\label{eq::tilde-pir-general-K}.
\end{align}

Hence, we have fully described our approximation to $\pi r$, once we specifically describe our methods for stochasticizing $G$. 
\end{remark}

\emph{The Perron-Frobenius approach:} Under \cref{assumption::G-irreducible}, $G$ has a unique positive eigenvalue $\lambda$ and associated positive row and column eigenvectors $\nu$ and $h$ respectively such that
\begin{align}
\sum_{x\in K}^{}\nu(x)h(x)=1.\label{eq::h-nu-integrates-to-1}
\end{align}
Since $K$ is of modest size, we can assume that $\lambda, \nu,$ and $h$ are computable. We now put
\begin{align*}
P_1(x,y) = \frac{G(x,y)h(y)}{\lambda h(x)}
\end{align*}
and note that $P_1=(P_1(x,y):x,y\in K)$ is a stochastic and irreducible matrix. Furthermore, if we set $\pi_1(x)=\nu(x)h(x)$, it is easily verified, in the presence of \cref{eq::h-nu-integrates-to-1}, that $\pi_1=(\pi_1(x):x\in K)$ is the unique equilibrium distribution of $P_1$. We can now put $\tilde \pi_K=\pi_1$ in \cref{eq::tilde-pir-general-K}, thereby yielding our first approximation $\tilde\pi_1(r)$ to $\pi r$. We note that the use of Perron-Frobenius normalization in the context of truncation is also discussed in Section 6.5 in \citet{senetaNonnegativeMatricesMarkov2006}.

\emph{The row-normalized approach:} For $x\in K$, let
\begin{align*}
n(x) = \sum_{y\in K}^{}G(x,y)
\end{align*}
be the row sum corresponding to $x$. Set
\begin{align*}
P_2(x,y) = \frac{G(x,y)}{n(x)}
\end{align*}
and note that $P_2=(P_2(x,y):x,y\in K)$ is an irreducible stochastic matrix. Hence, $P_2$ has a unique equilibrium distribution $\pi_2=(\pi_2(x):x\in K)$. By putting $\tilde \pi_K=\pi_2$, this yields a second approximation to $\pi_K$, and hence a second approximation $\tilde\pi_2(r)$ to $\pi r$ in view of \cref{eq::tilde-pir-general-K}.

For a non-negative function $w$, let
\begin{align*}
\kappa_i(w) = E_{\pi_i}\smashoperator{\sum_{j=0}^{(T_K\land T)-1}} w(X_j)
\end{align*}
for $i=1,2$. Then, the $i$'th approximation to $\pi r$ is given by $\tilde \pi_i(r)=\kappa_i(r)/\kappa_i(e)$. 

We now argue that both of our approximations to $\pi r$ are convergent under the sole condition that $X$ is irreducible and positive recurrent. 

\begin{proposition}\label{prop_convergence}
Suppose we fix the inner (finite) subset $K$ and consider a sequence $(A_n:n\geq 1)$ of finite subsets for which $A_n\supseteq K$ and $A_n\nearrow S$ as $n\rightarrow\infty$. Then $\tilde\pi_i(r)\rightarrow\pi r$ as $n\rightarrow\infty$ for $i=1,2$.
\end{proposition}
\begin{proof}
In this case, the random variable $T$ depends on $n$ (i.e. $T=T_n=\inf\{j\geq 1: X_j\in A_n^c\}$), and $T_K\land T_n\nearrow T_K$ a.s. as $n\rightarrow\infty$. It follows from the Monotone Convergence Theorem that
\begin{align}
E_x\smashoperator{\sum_{j=0}^{(T_n\land T_K)-1}}w(X_j)\nearrow E_x\sum_{j=0}^{T_K-1}w(X_j)\label{eq::Monotone-Convergence-Theorem-on-sum-of-w}
\end{align}
as $n\rightarrow\infty$. Similarly, the Monotone Convergence Theorem implies that $G_n = (G_n(x,y):x,y\in K)$ converges to $P_K$. Since $P_i$ and $P_K$ are irreducible, and $P_i$ converges to $P_K$ as $n\rightarrow\infty$, the equilibrium distribution of $P_i$ converges to that of $P_K$ as $n\rightarrow\infty$; see p.404 of \citet{schweitzerPerturbationTheoryFinite1968}. We therefore conclude from \cref{eq::Monotone-Convergence-Theorem-on-sum-of-w} that $\kappa_i(w)\rightarrow \kappa(w)$ as $n\rightarrow\infty$ for each non-negative $w$, thereby implying that $\tilde \pi_i(r)\rightarrow \pi r$ as $n\rightarrow\infty$ for $i=1,2$.
\end{proof}

\section{Computable Bounds for Equilibrium Expectations\label{sec::EqExpBounds}}
In this section, we derive computable lower and upper bounds on $\pi r$. We start by noting that if
\begin{align*}
\kappa(x,r) = E_x\sum_{j=0}^{T_K-1}r(X_j)
\end{align*}
for $x\in K$ and $r$ non-negative, then we obtain an easy lower bound for $\kappa(x,r)$ via 
\begin{align*}
\kappa(x,r)\geq \utilde\kappa(x,r),
\end{align*}
where 
\begin{align*}
\utilde \kappa(x,r) &= E_x\sum_{j=0}^{(T_K\land T)-1} r(X_j)\\
&=r_1(x) + \left(P_{12}(I-P_{22})^{-1}r_2\right)(x).
\end{align*}
To obtain an upper bound $\tilde \kappa(x,r)$ satisfying 
\begin{align*}
\kappa(x,r)\leq \tilde\kappa(x,r),
\end{align*}
it is evident from \cref{eq::block_linear_system} that we need upper bounds on $P_{13}\eta(\bar r)$ and $P_{23}\eta(\bar r)$. Recall that for $x\in A^c$,
\begin{align*}
(\eta(\bar r))(x) = \sum_{j=0}^{\infty}(B^j\bar r)(x),
\end{align*}
so $\eta(\bar r)$ involves the (potentially) infinite dimensional matrix $B$. Since direct numerical computations with $B$ are infeasible, we shall assume the existence of Lyapunov bounds (as generated by the analyst) to control the magnitude of $\eta(\bar r)$ and $\eta(\bar e)$. In particular, we shall assume:

\begin{assumption}\label{assumption::Lyapunov1} There exist known non-negative functions $g_1:K^c\rightarrow\R_+, g_2:K^c\rightarrow\R_+, h_1:A\rightarrow\R_+$ and $h_2:A\rightarrow \R_+$ such that 
\begin{subequations}
\begin{align}
\sum_{y\in K^c}^{}P(x,y)g_1(y)&\leq g_1(x)-r(x)\label{eq::Lyapunov-for-r},\\
\sum_{y\in K^c}^{}P(x,y)g_2(y)&\leq g_2(x)-e(x)\label{eq::Lyapunov-for-e}
\end{align}
\label{eq:LyapGlobal}
\end{subequations}
for $x\in K^c$, and 
\begin{align}
\sum_{y\in A^c}^{}P(x,y)g_i(y)\leq h_i(x)\label{eq::Pg-bound}
\end{align}
for $x\in A, i=1,2$. 
\end{assumption}

\begin{remark} The functions $g_1$ and $g_2$ are called (stochastic) Lyapunov functions associated with $\bar r$ and $\bar e$, respectively. We note that a sufficient condition for \cref{eq:LyapGlobal} involves extending the domain of the $g_i$'s to be the entire state space $S$, and then verifying that $(Pg_1)(x)\leq g_1(x)-r(x)$ and $(Pg_2)(x)\leq g_2(x)-1$ for $x\in K^c$. Typically, one first verifies these inequalities analytically by analyzing $Pg_i$ for $x$ sufficiently large, after which one can choose $K$ as the complement of the set for which the inequalities are known to hold. We discuss the construction of $K$ in greater detail in \cref{sec::NumericalExperiments}.
\end{remark}

The following result is well known (see, for example, \citet{meynMarkovChainsStochastic2009}, p.344), but we give a short proof to make our presentation as self-contained as possible.

\begin{proposition} \label{prop-Lyapunov} When $r$ is non-negative, \cref{eq::Lyapunov-for-r} implies that
\begin{align*}
\left(\eta(\bar r)\right)(x) \leq g_1(x)
\end{align*}
for $x\in A^c$. 
\end{proposition}

\begin{proof} Note that the Lyapunov inequality implies that 
\begin{align}
(Bg_1)(x)\leq g_1(x)-\bar r(x)\label{eq::Lyapunov-in-terms-of-B}
\end{align}
for $x\in K^c$. It follows that $Bg_1\leq g_1$. Because $B$ is non-negative, this implies that $B^{n+1}g_1\leq B^ng_1$ for $n\geq 0$, so that $B^ng_1\leq g_1$. Consequently, $B^ng_1$ is finite-valued for $n\geq 0$. 

The inequality \cref{eq::Lyapunov-in-terms-of-B} implies that
\begin{align*}
\bar r \leq g_1-Bg_1
\end{align*}
and hence
\begin{align}
B^k\bar r \leq B^kg_1-B^{k+1}g_1\label{eq::B-to-the-K-inequality}
\end{align}
for $k\geq 0$. Summing \cref{eq::B-to-the-K-inequality} over $0\leq k\leq n$, we conclude that
\begin{align*}
\sum_{k=0}^{n}B^k\bar r\leq g_1-B^{n+1}g_1\leq g_1.
\end{align*}
Sending $n\rightarrow\infty$, we obtain the result.
\end{proof}

If we write the function $h_1$ as a partitioned column vector $h_1^T=(h_{11}^T,h_{12}^T)$, we find that
\begin{align*}
\left(P_{13}\eta(\bar r)\right)(x) &= \sum_{y\in A^c}^{}P(x,y)(\eta(\bar r))(y)\\
&\leq \sum_{y\in A^c}^{}P(x,y)g_1(y)\\
&\leq h_{11}(x)
\end{align*}
for $x\in K$, and 
\begin{align*}
\left(P_{23}\eta(\bar r)\right)(x) &= \sum_{y\in A^c}^{}P(x,y)(\eta(\bar r))(y)\\
&\leq \sum_{y\in A^c}^{}P(x,y)g_{1}(y)\\
&\leq h_{12}(x)
\end{align*}
for $x\in A'$. Hence, in the presence of the Lyapunov bound, we find that
\begin{align*}
\kappa(x,r)\leq \tilde \kappa_1(x,r), 
\end{align*}
where
\begin{align*}
\tilde \kappa_1(x,r) = \utilde \kappa(x,r) + h_{11}(x) + (P_{12}(I-P_{22})^{-1}h_{12})(x).
\end{align*}

\begin{remark} The optimal choice of the function $h_1$ and the Lyapunov function $g_1$ is attained when \cref{eq::Lyapunov-for-r} and \cref{eq::Pg-bound} hold with equality. In that case
\begin{align*}
g_1(x) = E_x \sum_{j=0}^{T_K-1}r(X_j)
\end{align*}
for $x\in K^c$,
\begin{align*}
h_{11}(x) &= E_x \sum_{j=1}^{T_K-1}r(X_j)I(X_1\in A^c),
\end{align*}
for $x\in K$, and 
\begin{align*}
h_{12}(x) &= E_x \sum_{j=T}^{T_K-1} r(X_j)I(X_1\in A', T<T_K)
\end{align*}
for $x\in A'$. With $g_1$ and $h_1$ so selected, 
\begin{align*}
\kappa(z,r) = \tilde \kappa_1(z,r), 
\end{align*}
so our upper bound on $\kappa(r)$ is exact. We therefore expect that our bounds can be quite tight when $g_1$ and $h_1$ are chosen well.
\end{remark}

If $K=\{z\}$ (i.e. $K$ is a singleton), we therefore have both upper and lower bounds on $\kappa(r)$, namely
\begin{align*}
\utilde \kappa(z,r)\leq \kappa(r)\leq \tilde \kappa_1(z,r). 
\end{align*}

The quantity $\kappa(e)$ can be similarly bounded via the functions $g_2$ and $h_2$. In particular, 
\begin{align*}
\utilde \kappa(x,e) \leq \kappa(x,e) \leq \tilde \kappa_2(x,e),
\end{align*}
where 
\begin{align*}
 \tilde\kappa_2(x,e) = \utilde \kappa(x,e) + h_{21}(x) + (P_{12}(I-P_{22})^{-1}h_{22})(x),
 \end{align*}
 with $h_2^T=(h_{21}^T, h_{22}^T)$. Hence if $K=\{z\}$, we obtain our desired bounds on $\pi r$, namely
\begin{align}
\frac{\utilde \kappa(z,r)}{\tilde \kappa_2(z,e)} \leq \pi r\leq \frac{\tilde \kappa_1(z,r)}{\utilde \kappa(z,e)}.\label{eq::singleton-bounds}
\end{align}

\begin{remark}\label{remark_single_pair_rem3} Note that if we find $g$ and $h$ for which
\begin{align*}
\sum_{y\in K^c}^{}P(x,y)g(y) \leq g(x)-(r(x)\lor 1)
\end{align*}
for $x\in K^c$, and
\begin{align*}
\sum_{y\in A^c}^{} P(x,y)g(y)\leq h(x)
\end{align*}
for $x\in A$, then we can set $g_1=g_2=g$ and $h_1=h_2=h$ in $\cref{eq:LyapGlobal}$ and $\cref{eq::Pg-bound}$. Of course, this choice, while easier to implement, will typically give worse bounds than obtained with tailored choices for the $g_i$'s and $h_i$'s.
\end{remark}

\begin{remark} When $X$ is irreducible and geometrically ergodic, there typically exists a Lyapunov function $v:S\rightarrow[1,\infty), \alpha<1$, and $c_v\in\R_+$ such that
\begin{align*}
\sum_{y\in S}^{}P(x,y)v(y)\leq \alpha v(x)+c_vI(x\in K)
\end{align*}
for $x\in S$; see \citet{meynMarkovChainsStochastic2009}, p.393. In this case, we can take $g_1=g_2=v$ on $K^c$, provided that $r(x)\leq (1-\alpha) v(x)$ for $x\in S$.
\end{remark}

We now turn to obtaining bounds on $\kappa(r)$ and $\kappa(e)$ when $K$ is not a singleton. Such an extension is important, because the relaxation of $K$ to non-singleton sets makes it significantly easier to construct the Lyapunov functions $g_1$ and $g_2$. Note that the contribution of the equilibrium probabilities $(\pi_K(x):x\in K)$ to $\kappa(r)$ and $\kappa(e)$ now needs to be explicitly considered, in order to generate upper and lower bounds on $\pi r$. In the remainder of this section, we show how to compute such bounds. 

In particular, note that $\pi_K\in \mathcal{P}$, where $\mathcal{P}$ is the set of probabilities $\phi$ on $K$ (encoded as row vectors) such that
\begin{align}
\phi(y) = \sum_{x\in K}^{}\phi(x)\rho(x,y)\label{eq::phi-is-stationary-distribution-of-rho}
\end{align}
for $y\in K$ for some stochastic matrix $(\rho(x,y):x,y\in K)$, where
\begin{align}
\rho(x,y)\geq G(x,y)\label{eq::G-lower-bounds-rho}
\end{align}
for $x,y\in K$. We can now obtain upper and lower bounds on $\kappa(r)$ via
\begin{align}
\min_{\phi\in \mathcal{P}}\phi\utilde\kappa(r)\leq \kappa(r)\leq \max_{\phi\in \mathcal{P}}\phi\tilde\kappa_1(r),\label{eq::LP-bounds-on-sums-of-rewards}
\end{align}
where, for a non-negative function $w$, $\tilde\kappa_i(w)$ and $\utilde\kappa(w)$ are the column vectors $(\tilde\kappa_i(x,w):x\in K)$ and $(\utilde\kappa(x,w):x\in K)$, respectively. 

It turns out that there is a powerful vehicle available for characterizing $\mathcal{P}$ that allows us to easily compute the bounds appearing in \eqref{eq::LP-bounds-on-sums-of-rewards}. In particular, for $x,y\in K$, let
\begin{align}
\tau_x(y) = \frac{(I-G)^{-1}(x,y)}{\sum_{w\in K}^{}(I-G)^{-1}(x,w)},\label{eq::ImG-distributions}
\end{align}
so that $\tau_x=(\tau_x(y):y\in K)$ is a probability on $K$, for each $x\in K$. \citet{courtoisBoundsPositiveEigenvectors1984} establish that $\mathcal{P}$ is identical to the family of mixtures of the $\tau_x$'s, so that
\begin{align}
\mathcal{P} = \{\phi: \phi = \sum_{x\in K}^{}\lambda(x) \tau_x \text{ where } (\lambda(x):x\in K) \text{ is a probability on $K$}\};\label{eq::CS-representation}
\end{align}
see their Theorem 2 and also \citet{https://doi.org/10.48550/arxiv.2208.03446} for a generalization to continuous state space. In view of \eqref{eq::CS-representation}, it follows that
\begin{align*}
\max_{\phi \in \mathcal{P}}\phi \tilde \kappa_1(r) = \max_{x\in K} \tau_x\tilde \kappa_1(r)
\end{align*}
and 
\begin{align*}
\min_{\phi\in\mathcal{P}} \phi \utilde \kappa(r) = \min_{x\in K}\tau_x\utilde\kappa (r).
\end{align*}
Similarly to \eqref{eq::singleton-bounds}, we are then led to the following lower and upper bounds on $\pi r$:
\begin{align}
\frac{\min_{x\in K}\tau_x \utilde\kappa(r)}{\max_{x\in K}\tau_x \tilde\kappa_x(e)} \leq \pi r\leq \frac{\max_{x\in K}\tau_x \tilde \kappa_1(r)}{\min_{x\in K} \tau_x \utilde \kappa(e)}.\label{eq::CS-lower-and-upper-bounds}
\end{align}
Since our approximation $\tilde \pi_i(r)$ from \cref{sec::Approximation} satisfies
\begin{align*}
\tilde\pi_i(r) = \frac{\pi_i \utilde \kappa(r)}{\pi_i \utilde \kappa(e)},
\end{align*}
it is evident that $\tilde \pi_2(r)$ is also bracketed by the same lower and upper bound as is $\pi r$ in \eqref{eq::CS-lower-and-upper-bounds}. (Note that $P_2$ is greater than or equal to $G$. Since $P_1$ can fail to be component wise greater than or equal to $G$, this bound may not be valid for $P_1$.) It is therefore evident that
\begin{align}
|\tilde \pi_2(r)- \pi r| \leq \frac{\max_{x\in K}\tau_x\tilde\kappa_1(r)}{\min_{x\in K}\tau_x\utilde \kappa(e)} - \frac{\min_{x\in K} \tau_x\utilde \kappa(r)}{\max_{x\in K}\tau_x\tilde\kappa_2(e)}.\label{eq::CS-approximation-error-bound}
\end{align}
Hence, we arrive both at lower and upper bounds on $\pi r$ itself, as well as a bound on the approximation error $|\tilde \pi_2(r)-\pi r|$. 

In the discussion of our numerical experiments in \cref{sec::NumericalExperiments}, we refer to the equilibrium expectation bound \eqref{eq::CS-lower-and-upper-bounds} as the ``minorization equilibrium expectation bound'' (because it relies on the fact that $G$ minorizes $P_K$, $G\leq P_K$).

\begin{remark} It can be shown, by example, that $P_1$ can fail to lie in $\mathcal{P}$. Consequently, the bound \eqref{eq::CS-approximation-error-bound} may fail to be valid. We consider a perturbation-based bound suitable for $\pi_1$ in the next section.
\end{remark}

\begin{remark} \label{remark::convergence_of_LP_bounds} We note that when we have a sequence $(A_n:n\geq 1)$ of truncation sets $A_n\nearrow S$ as $n\rightarrow\infty$ with $K$ fixed, our upper and lower bounds on $\pi r$ can be guaranteed to converge to $\pi r$ as $n\rightarrow\infty$ under suitable conditions. In that setting, $G_n \nearrow P_K$ as $n\rightarrow\infty$, and $\mathcal{P}$ decreases to $\{\pi_K\}$. Furthermore, the Monotone Convergence Theorem guarantees that $\utilde\kappa(x,r)\nearrow\kappa(x,r)$ as $n\rightarrow\infty$. We also observe that if $E_xg_i(X_1)<\infty$ for $x\in K$, then when $h_i$ can be computed exactly,
\begin{align*}
h_{i1}(x)=E_xg_i(X_1)I(X_1\in A_n^c)\rightarrow 0
\end{align*}
as $n\rightarrow\infty$ via the Dominated Convergence Theorem. Also, if $\pi g_i<\infty$, 
\begin{align*}
(P_{12}(I-P_{22})^{-1}h_{i2})(x) = E_x g_i(X_{T_n})I(T_n<T_K)\rightarrow 0
\end{align*}
as $n\rightarrow\infty$, since $g_i(X_{T_n})I(T_n<T_K)$ is then dominated by $\sum_{j=0}^{T_K-1}g_i(X_j)$, which is integrable, on account of \cref{eq::RegenerativeRepresentation-pir} and $\pi g_i<\infty$. This guarantees that $\tilde \kappa_1(z,r)\searrow \kappa(z,r)$ and $\tilde \kappa_2(z,e)\searrow \kappa(z,e)$ as $n\rightarrow\infty$.

We note that a simple sufficient condition for $\pi g_i<\infty$ is the existence of a further non-negative Lyapunov function $g'_i$ and $c_i'\in \R$ such that
\begin{align*}
(Pg_i')(x)\leq g_i'(x)-g_i(x)+c_i'
\end{align*}
for $x\in S$; see \citet{glynnBoundingStationaryExpectations2008}.
\end{remark}

\section{Computable Total Variation Error Bounds}\label{sec::TV-bounds}
In this section, we apply the ideas of \cref{sec::EqExpBounds} to the development of a computable upper bound on the total variation distance between the distribution associated with our approximation $\tilde \pi_i(r)$ and the exact underlying equilibrium distribution $\pi$. We start by noting that $\tilde \pi_i(\cdot)$ induces a probability $\pi_i^*$ on $S$ via the formula $\pi_i^*(x)=\tilde\pi_i(e_x)$ where $e_x$ is the non-negative function for which $e_x(y)$ is 1 if $y=x$ and 0 otherwise. The \emph{total variation norm} distance between $\pi_i^*$ and $\pi$ associated with the (non-negative) envelope function $r$ is then given by
\begin{align*}
\norm{\pi_i^*-\pi}_{r}&=\sup_{|f|\leq r}\left|\sum_{x\in S}^{}\pi_i^*(x)f(x)-\sum_{x\in S}^{}\pi(x)f(x)\right|\\
&=\sup_{|f|\leq r}\left|\pi_i^*f-\pi f\right|.
\end{align*}
For a function $f$ of mixed sign, we can express it as $f=f^+-f^-$, where $f^+(x)=\max(f(x),0)$ and $f^-(x)=\max(-f(x),0)$. Then,
\begin{align*}
\norm{\pi_i^*-\pi}_{r}&=\sup_{|f|\leq r}\left|\pi_i^*f^+-\pi f^+-(\pi_i^*f^--\pi f^-)\right|\\
&\leq 2\sup_{0\leq f\leq r}\left|\pi_i^*f-\pi f\right|.
\end{align*}
To bound the quantity, we start with the simplest setting, in which $K=\{z\}$, so that $\pi_i=\pi_K$ for $i=1,2$. Then, with $f$ non-negative, 
\begin{align*}
\pi f-\tilde\pi_i(f) &= \frac{\kappa(z,f)}{\kappa(z,e)} - \frac{\utilde\kappa(z,f)}{\utilde\kappa(z,e)}\\
&\leq \frac{\kappa(z,f)-\utilde \kappa(z,f)}{\utilde \kappa(z,e)}\\
&\leq \frac{\kappa(z,r)-\utilde\kappa(z,r)}{\utilde\kappa(z,e)}\\
&\leq \frac{\beta_1(z)}{\utilde\kappa(z,e)},
\end{align*}
where
\begin{align*}
\kappa(x,r)-\utilde\kappa(x,r)&\leq \tilde\kappa_1(x,r) - \utilde\kappa(x,r)\\
&\leq h_{11}(x)+(P_{12}(I-P_{22})^{-1}h_{12})(x)\\
&\overset{\Delta}{=}\beta_1(x)
\end{align*}
for $x\in K$.

Also,
\begin{align*}
\tilde\pi_i(f) -\pi f &= \frac{\utilde\kappa(z,f)\left[\kappa(z,e)-\utilde\kappa(z,e)\right]+\utilde\kappa(z,e)\left[\utilde\kappa(z,f)-\kappa(z,f)\right]}{\utilde\kappa(z,e)\kappa(z,e)}\\
&\leq \frac{\utilde\kappa(z,f)\left[\kappa(z,e)-\utilde\kappa(z,e)\right]}{\utilde\kappa(z,e)\kappa(z,e)}\\
&\leq \frac{\utilde\kappa(z,f)\left[\kappa(z,e)-\utilde\kappa(z,e)\right]}{\utilde\kappa(z,e)^2}\\
&\leq \frac{\utilde\kappa(z,r)\left[\kappa(z,e)-\utilde\kappa(z,e)\right]}{\utilde\kappa(z,e)^2}\\
&\leq \tilde\pi_i(r) \frac{\beta_2(z)}{\utilde\kappa(z,e)},
\end{align*}
where
\begin{align*}
\kappa(x,e) -\utilde\kappa(x,e) &\leq \tilde\kappa(x,e)-\utilde\kappa(x,e)\\
&\leq h_{21}(x)+\left(P_{12}(I-P_{22})^{-1}h_{22}\right)(x)\\
&\overset{\Delta}{=}\beta_2(x)
\end{align*}
for $x\in K$. 

We therefore obtain the following computable error bound on the weighted total variation distance between $\pi_i^*$ and $\pi$.

\begin{theorem}\label{thm:tv_bound_singleton} Suppose $K=\{z\}$ and assume that \cref{assumption::Lyapunov1} is in force. Then,
\begin{align*}
\norm{\pi_i^*-\pi}_{r} \leq 2\max\left(\frac{\beta_1(z)}{\utilde\kappa(z,e)}, \tilde\pi_i(r)\cdot\frac{\beta_2(z)}{\utilde\kappa(z,e)}\right).
\end{align*}
\end{theorem}
We now extend the bound of \cref{thm:tv_bound_singleton} to the setting in which $K$ is not a singleton. Let $\Delta_i = \sum_{x\in K}^{}|\pi_i(x)-\pi_K(x)|=\norm{\pi_i-\pi_K}_{1}$, where $\norm{u}_{1} \overset{\Delta}{=}\sum_{x\in K} |u(x)|$ for a generic row vector $u=(u(x):x\in K)$. In this case, for $f$ non-negative with $f\leq r$,
\begin{align*}
&\left|\tilde\pi_i(r)-\pi f\right| \\
&\quad= \left|\frac{\pi_i\utilde\kappa(f)}{\pi_i\utilde\kappa(e)}- \frac{\pi_K \kappa(f)}{\pi_K\kappa(e)}\right|\\
&\quad\leq \frac{|\pi_i \utilde\kappa(f) \left(\pi_K\utilde\kappa(e) + \pi_K (\kappa(e)-\utilde\kappa(e))\right) - \pi_i\utilde\kappa(e)\left(\pi_K\utilde\kappa(f) + \pi_K(\kappa(f)-\utilde\kappa(f)\right)}{\pi_i\utilde\kappa(e)\pi_K\kappa(e)}\\
&\quad= \biggr(|\pi_i\utilde\kappa(f)\left[\left(\pi_K-\pi_i\right)\kappa(e) + \pi_i\left(\kappa(e)-\utilde\kappa(e)\right)\right] \\
&\qquad\qquad -\pi_i\utilde\kappa(e) \left((\pi_K-\pi_i)\kappa(f)+\pi_i(\kappa(f)-\utilde\kappa(f))\right)\biggr)/(\pi_i\utilde\kappa(e)\pi_K\kappa(e))\\
&\quad\leq \frac{\pi_i\utilde\kappa(f)\left(\Delta_i \norm{\kappa(e)}_{\infty}+\pi_i\beta_2\right)}{\pi_i\utilde\kappa(e)\cdot\pi_K\kappa(e)} + \frac{\pi_i\utilde\kappa(e)\left(\Delta_i\norm{\kappa(f)}_{\infty}+\pi_i\beta_1\right)}{\pi_i\utilde\kappa(e)\cdot\pi_K\kappa(e)}\\
&\quad\leq \frac{\left(\pi_i\beta_1+\tilde\pi_i(r)\pi_i\beta_2\right)}{\ell_i(e)} + \Delta_i \left(\frac{\tilde\pi_i(r) \norm{\tilde\kappa_2(e)}_{\infty}+\norm{\tilde\kappa_1(r)}_{\infty}}{\ell_i(e)}\right)\\
&\quad\overset{\Delta}{=}\epsilon_i(r).
\end{align*}
We therefore arrive at our weighted total variation bound for the setting in which $K$ is not a singleton.
\begin{theorem}\label{thm:tv_bound_general}
When \cref{assumption::G-irreducible} and \cref{assumption::Lyapunov1} are in force, then
\begin{align*}
\norm{\pi_i^*-\pi}_{r} \leq 2\epsilon_i(r).
\end{align*}
\end{theorem}

As noted in \cref{remark::convergence_of_LP_bounds}, the weighted total variation bound converges to zero under suitable conditions (i.e. $\pi g_i<\infty$ for $i=1,2$) as the truncation set $A_n$ increases to $S$.

It remains only to discuss how $\Delta_i$ can be numerically bounded. Since $\pi_2,\pi_K\in \mathcal{P}$, the Courtois-Semal result exploited in \cref{sec::EqExpBounds} establishes that
\begin{align*}
\pi_2-\pi_K = \sum_{x\in K}^{}(\lambda_2(x)-\lambda_K(x))\tau_x
\end{align*}
for appropriately chosen probabilities $(\lambda_2(x):x\in K)$ and $(\lambda_K(x):x\in K)$. Hence,
\begin{align}
\Delta_2 \leq \max \{\sum_{y\in K} |\sum_{x\in K}^{}(q_1(x)-q_2(x))\tau_x(y)|: q_1,q_2 \text{ are probabilities on }K\}.\label{eq::CS-bound-on-Deltai}
\end{align}
In \citet{https://doi.org/10.48550/arxiv.2208.03446}, it is argued that the right-hand side of \eqref{eq::CS-bound-on-Deltai} is further upper bounded by $\max_{x,y\in K} \sum_{z\in K}^{}|\tau_x(z)-\tau_y(z)|$, which is easily computable.

We now provide an alternative approach suitable for bounding $\Delta_1$. It is well known that
\begin{align}
\pi_1-\pi_K = \pi_K(P_1-P_K)F_i.\label{eq::PerturbationInTermsofFundamentalMatrix}
\end{align}
where $F_1$ is the \emph{fundamental matrix} associated with $P_1$, given by $F_1=(I-P+\Pi_1)^{-1}$ (with $\pi_1$ the matrix in which each row is identical to $\pi_1$); see, for example, \citet{liuPerturbationBoundsStationary2012}. For a matrix $W=(W(x,y):x,y\in K)$  define the norm $\norm{W}_{\infty}=\max_{x\in K}\sum_{y\in K}|W(x,y)|$. Then
\begin{align}
\norm{\pi_1-\pi_K}_{1} &\leq \norm{\pi_K}_{1}\cdot\norm{(P_i-P_K)F_i}_{\infty}\nonumber\\
&=\norm{\left(P_i-P_K\right)F_i}_{\infty},\label{eq::ptb-bound-on-pi-i-and-piK-in-terms-of-P-difference}
\end{align}
using standard properties of row and matrix norms (see, for example, \citet{golubMatrixComputations4th2013}).
We now write $P_K=G+N$, where $N=(N(x,y):x,y\in K)$ has non-negative entries defined by $N(x,y)=P_K(x,y)-G(x,y)$ for $x,y\in K$. Note that
\begin{align*}
\norm{N}_{\infty}&=\max_{x\in K}\sum_{y\in K}^{}N(x,y)\\
&= \max_{x\in K}\left(1-\sum_{y\in K}^{}G(x,y)\right)\\
&=1-\min_{x\in K}n(x) \overset{\Delta}{=} 1-\delta.
\end{align*}
Consequently, it follows from \eqref{eq::ptb-bound-on-pi-i-and-piK-in-terms-of-P-difference} that 
\begin{align}
\Delta_1&\leq \norm{\left(P_1-G-N\right)F_1}_{\infty}\nonumber\\
&\leq \norm{(P_1-G)F_1}_{\infty}+ (1-\delta)\norm{F_1}_{\infty}\label{eq::perturbation-based-total-variation-bound}
\end{align}
The principal computational effort required to calculate this bound on $\Delta_1$ involves the calculation of the fundamental matrix $F_1$. 
\begin{remark} In principle, the above perturbation argument could also have been applied to obtain a bound on $\Delta_2$. Our computational experience indicates that the bound associated with \eqref{eq::CS-bound-on-Deltai} is typically tighter for $\Delta_2$ than is the bound based on \eqref{eq::perturbation-based-total-variation-bound}. In fact, the bound \eqref{eq::CS-bound-on-Deltai} holds whenever $\tilde G$ pointwise dominates $G$ and the bound \eqref{eq::perturbation-based-total-variation-bound} holds for any irreducible stochastic $\tilde G$. Therefore practitioners who wish to use different approximations than the ones we propose ($\pi_1$ and $\pi_2$) are free to use the bounds \eqref{eq::CS-bound-on-Deltai} and \eqref{eq::perturbation-based-total-variation-bound}, provided they confirm the relevant conditions are in force.
\end{remark}

In the discussion of our numerical experiments in \cref{sec::NumericalExperiments}, we refer to the total variation bound based on \eqref{eq::CS-bound-on-Deltai} as the ``minorization TV bound,'' and we refer to the total variation bound based on \eqref{eq::perturbation-based-total-variation-bound} as the ``perturbation TV bound''.

\section{A Probabilistic Perspective on the Row-normalized Approximation}\label{sec::probabilistic-perspective-on-row-normalization}

We now return to the second approximation $\pi_2^*(\cdot)$ to $\pi$ that is based on an approximation to $\pi_K$ that derives from stochasticizing $G$ via row normalization. We shall now show that the stochastic matrix $P_2$ that arises in that setting has a natural probabilistic interpretation, and that it is associated with a corresponding stochastic matrix approximation to the entire stochastic matrix $P$ having state space $S'$ for which $K\subseteq S'\subseteq A\subseteq S$. 

For $x\in A$, let
\begin{align*}
u(x)=P_x(T_K<T).
\end{align*}
We note that $u(x)=n(x)=\sum_{y\in K}^{}G(x,y)$ for $x\in K$. Furthermore, for $x\in A$,
\begin{align*}
u(x)=P_x(X_1\in K)+\sum_{y\in A'}^{}P(x,y)u(y).
\end{align*}
Let $S'=\{z\in A: u(z)>0\}$, and set
\begin{align*}
R(x,y) = \begin{cases} P(x,y)/u(x), & x\in S', y\in K\\
P(x,y)u(y)/u(x), &x\in S',y\in S'-K.
\end{cases}
\end{align*}
Then, $R=(R(x,y):x,y\in S')$ is a stochastic matrix. The definition of $u(\cdot)$ implies that $K$ is a reachable set from each $x\in S'$ under $R$ (i.e. for each $x\in S'$, there exists $y_1,y_2,...,y_n$ in $S'$ with $y_n\in K$ for which $R(x,y_1)R(y_1,y_2)...R(y_{n-1},y_n)>0$.). Given \cref{assumption::G-irreducible}, it follows that $S'$ has a single closed communicating class $S''$ under $R$, with $K\subseteq S''$. In particular, $ R=(R(x,y):x,y\in S'')$ is an irreducible transition matrix.

For $x\in S'$ and $y_1, y_2,...,y_n\in S'-K$, 
\begin{align*}
&P_x(X_0=x,X_1=y_1,...,X_n=y_n, n<T_K<T)\\
&\ \ =P(x,y_1)...P(y_{n-1},y_n)u(y_n)\\
&\ \ =\left(R(x,y_1)u(x)/u(y_1)\right)\left(R(y_1,y_2)u(y_1)/u(y_2)\right)...\left(R(y_{n-1},y_n)u(y_{n-1})/u(y_n)\right)u(y_n)\\
&\ \ =u(x)R(x,y_1)...R(y_{n-1},y_n)\\
&\ \ =u(x)\tilde P_x(X_0=x, X_1=x_1,...,X_n=x_n),
\end{align*}
where $\tilde P_x(\cdot)$ is the probability on the path-space of $X$ under which $X$ evolves according to the transition matrix $R$. Similarly, we find that for $x\in S'$,$y\in K$, and $y_1,...,y_n\in S'-K$,
\begin{align*}
P_x(X_0=x,& X_1=y_1,...,X_n=y_n, X_{n+1}=y) = \\
&\qquad u(x)\tilde P_x(X_0=x, X_1=y_1, ...,X_n=y_n,X_{n+1}=y).
\end{align*}
It follows that for $x\in S'$, 
\begin{align*}
\tilde P_x((X_0,X_1,...,X_{T_K})\in \cdot)= P_x\left((X_0,X_1,...,X_{T_K})\in \cdot | T>T_K\right),
\end{align*}
so that the Markov chain having transition matrix $R$ exactly describes the dynamics of the original Markov chain having transition matrix $P$, conditional on $T>T_K$. Furthermore, we note that for $x,y\in K$,
\begin{align*}
\tilde P_x (X_{T_K}=y) &= \frac{P_x(X_{T_K}=y, T>T_K)}{u(x)}\\
&=G(x,y)/u(x)\\
&=P_2(x,y),
\end{align*}
so that the row-normalized stochastic matrix $P$ is precisely the transition matrix for the process on $K$ (or, equivalently, the censored Markov chain) associated with the Markov chain $X$ on $S'$ having transition matrix $R$. In other words, $P_2$ is an exact truncation matrix for the process on $K$ associated with the stochastic matrix approximation $R$ to $P$ that arises by conditioning the dynamics of the $P$-chain so that $T_K$ occurs before the exit $T$ to $A^c$. 

In view of the natural connection between $R$ and $P$, this suggests a third approximation to $\pi$, namely the (unique) equilibrium distribution $\pi_3^*$ associated with the transition matrix $R$, suitably extending $\pi_3^*$ from $S''$ to $S$ by setting $\pi_3^*(z)=0$ for $z\in S-S''$.

\begin{remark} The function $u(\cdot)$ can be computed, since 
\begin{align*}
u(x)=\left((I-P_{22})^{-1}P_{21}e_1\right)(x)
\end{align*}
for $x\in A'$ and
\begin{align*}
u(x)=(P_{11}e_1)(x)+(P_{12}(I-P_{22})^{-1}P_{21}e_1)(x)
\end{align*}
for $x\in K$. (We note that $u(y)=0$ is a possibility, and hence $S'$ may be a proper subset of $A$.) Once $u(\cdot)$ has been computed, $R$ becomes available and $S'$ and $S''$ can be identified.
\end{remark}

As for our first two approximations, it is of interest to know if $\pi_3^*$ converges to $\pi$ as the truncation set $A$ gets large. In particular, as in the discussion of \cref{sec::Approximation}, let $(A_n:n\geq 1)$ be a sequence of subsets containing $K$ for which $A_n\nearrow S$ as $n\rightarrow\infty$. Let $u_n(x)=P_x(T_K<T_n)$ and let $\pi_{3n}^*(\cdot)$ be the equilibrium distribution associated with the transition matrix $R_n$ having state space $S_n''$. The irreducibility and positive recurrence of the Markov chain $X$ under the transition matrix $P$ then implies our next result:

\begin{theorem} For each $x\in S$, $\pi_{3n}^*(x)\rightarrow\pi(x)$ as $n\rightarrow\infty$. 
\end{theorem}
\begin{proof} Any weakly convergent subsequence of $(\pi_{3n}(\cdot):n\geq 1)$ must necessarily be a stationary distribution of $P$; see, for example, \citet{karrWeakConvergenceSequence1975}. Since $P$ has a unique equilibrium distribution, it follows that $\pi$ is the unique weakly convergent limit point of $(\pi^*_{3n}(\cdot):n\geq 1)$. It is evident that the proof is therefore complete if we show that $(\pi_{3n}^*(\cdot):n\geq 1)$ is a tight sequence of probabilities on $S$. 

For $\epsilon>0$, choose a finite subset $K(\epsilon)\supseteq K$ for which $\pi(K(\epsilon)^c)<\epsilon$. Then \cref{eq::GeneralizedRegenerativeRepresentation} implies that 
\begin{align}
\frac{E_{\pi_K}\sum_{j=0}^{T_K-1}I(X_j\in K(\epsilon)^c)}{E_{\pi_K}T_K}<\epsilon.\label{eq::tightness-for-pi3}
\end{align}
Note that if $\tilde E_x^n(\cdot)$ is the expectation on the path-space of $X$ under which $X_0=x$ and $X$ evolves according to $R_n$, we can similarly apply \cref{eq::GeneralizedRegenerativeRepresentation} to the $R_n$-chain, thereby yielding the identity
\begin{align}
\pi_{3n}^*\left(K(\epsilon)^c\right)= \frac{\tilde E_{\pi_K^n}^n \sum_{j=0}^{T_{K-1}}I(X_j\in K(\epsilon)^c)}{\tilde E_{\pi_K^n}^n T_K},\label{eq::ratio-representation-for-probability-pi3-is-in-Kc}
\end{align}
where $\pi_K^n$ is the stationary distribution of the process on $K$ associated with the $R_n$ chain. But for $x\in K$,
\begin{align}
\tilde E_x^n \sum_{j=0}^{T_{K}-1}I(X_j\in K(\epsilon)^c) &= E_x\left[\sum_{j=0}^{T_K-1}I(X_j\in K(\epsilon)^c)\biggr|T_K<T_n\right]\nonumber\\
&= \frac{E_x\sum_{j=0}^{T_K-1}I(X_j\in K(\epsilon)^c)I(T_K<T_n)}{P_x(T_K<T_n)}.\label{eq::numerator-of-ratio-representation-for-probability-pi3-is-in-Kc}
\end{align}
The Monotone Convergence Theorem applies to the numerator and denominator of \cref{eq::numerator-of-ratio-representation-for-probability-pi3-is-in-Kc} when one sends $n$ to infinity. Similarly, $\tilde E_x^nT_K\rightarrow E_xT_K$ as $n\rightarrow\infty$ via the Monotone Convergence Theorem. Finally, we proved in \cref{sec::Approximation} that $\pi_K^n(x)\rightarrow \pi_K(x)$ as $n\rightarrow\infty$. It follows from \cref{eq::ratio-representation-for-probability-pi3-is-in-Kc} and \cref{eq::numerator-of-ratio-representation-for-probability-pi3-is-in-Kc} that 
\begin{align*}
\overline\lim_{n\rightarrow\infty}\pi_{3n}^*(K(\epsilon)^c)<\epsilon,
\end{align*}
establishing the tightness and completing the argument.
\end{proof}

\section{The Connection of the Exit Approximation to the Equilibrium Approximation}\label{sec::exit_approx}

In this section, we will introduce what is perhaps the simplest of all convergent truncation approximations to the equilibrium distribution, namely what we shall call the \emph{exit approximation} to $\pi$. We will then establish a connection between the exit approximation and the approximations $\pi_i^*$ when $K$ is a singleton $\{z\}$.

As in \cref{sec::Approximation}, let $A$ be the truncation set (chosen as large as possible, subject to computational tractability), and let $T=\inf\{n\geq 0:X_n\in A^c\}$ be the exit time from $A$. For a given state $z\in A$, the exit approximation is given by 
\begin{align*}
\pi_e^*(\cdot) = \frac{E_z\sum_{j=0}^{T-1}I(X_j\in \cdot)}{E_z T}.
\end{align*}
With the same notation as in \cref{sec::Approximation}, let
\begin{align*}
H = \begin{pmatrix}
P_{11} & P_{12}\\
P_{21} & P_{22}
\end{pmatrix},
\end{align*}
and put $r'^T=(r_1^T, r_2^T), e'^T=(e_1^T, e_2^T)$. Then, if $H$ is irreducible, $(I-H)^{-1}$ exists and $\pi_e^*(\cdot)$ can be computed via
\begin{align*}
\pi_e^*(x)= \frac{\left((I-H)^{-1}e_x\right)(z)}{\left((I-H)^{-1}e'\right)(z)}.
\end{align*}
This approximation to $\pi$ is discussed in \citet{senetaComputingStationaryDistribution1980} (see (2) in that paper).

\begin{remark} Note that $\pi_e^*(x) = \nu_e(x)/\sum_{y}^{}\nu_e(y)$ for $x\in A$, where
\begin{align*}
\nu_e(x) = \left((I-H)^{-1}e_x\right)(z). 
\end{align*}
Then, $\nu_e=(\nu_e(x):x\in A)$ is the unique solution of the linear system
\begin{align}
\nu_e = \delta_z + \nu_eH,\label{eq::exit_apx_linear_system}
\end{align}
where $\delta_z$ is the row vector in which $\delta_z(x)=1$ if $x=z$ and 0 otherwise. Hence, the exit approximation can be computed directly from the solution $\nu_e$ to the linear system \cref{eq::exit_apx_linear_system}.
\end{remark}

We now relate the exit approximation to the approximations of \cref{sec::Approximation}, where we now choose $K=\{z\}$. Let $\tau=\inf\{n\geq 1: X_n=z\}$ be the first return time to $z$ and note that if $r$ is a non-negative function
\begin{align*}
E_{z} \sum_{j=0}^{T-1} r(X_j) &= E_{z} \sum_{j=0}^{T-1} r(X_j) I(T<\tau) + E_z \sum_{j=0}^{T-1}r(X_j)I(\tau<T)\\
&=E_z \sum_{j=0}^{T-1}r(X_j) I(T<\tau) + E_x \sum_{j=0}^{\tau-1}r(X_j)I(\tau<T) + E_z \sum_{j=\tau}^{T-1}r(X_j)I(\tau<T)\\
&=E_z \sum_{j=0}^{(T\land \tau)-1}r(X_j) + E_z \sum_{j=0}^{T-1}r(X_j)\cdot P_z(\tau<T),
\end{align*}
from which we conclude that 
\begin{align*}
E_z\sum_{j=0}^{T-1}r(X_j) = \frac{E_z \sum_{j=0}^{(T\land \tau)-1}r(X_j)}{P_z(T<\tau)}.
\end{align*}
Consequently, if $K=\{z\}$, 
\begin{align}
\pi_e^*(\cdot) = \frac{E_z \sum_{j=0}^{(T\land T_K)-1}I(X_j\in \cdot)}{E_z(T\land T_K)}, \label{eq::singleton-exit-approx}
\end{align}
which precisely coincides with our approximation \cref{eq::tilde-pir-singleton-case}.

While $\pi_i^*$ and $\pi_e^*$ are mathematically equivalent under $K=\{z\}$, it should be noted that $\pi_e^*$ is most naturally computed in terms of $(I-H)^{-1}$ whereas $\pi_i^*$ is most naturally computed in terms of $(I-P_{22})^{-1}$. When $A$ is large, the rows of $H$ corresponding to states in $A'$ have row sums closer to 1 then those of $P_{22}$. This may have numerical implications.

\section{Extension to Markov Jump Processes}\label{sec::MarkovJumpProcess}
We now extend our analysis to the setting of an $S$-valued irreducible positive recurrent Markov jump process $(X(t):t\geq 0)$ having rate matrix $Q=(Q(x,y):x,y\in S)$, with either finite or countably infinite state space $S$. Set $\lambda(x)\overset{\Delta}{=}-Q(x,x)$ for $x\in S$. The appropriate continuous-time analog of the discrete time Lyapunov conditions \cref{eq:LyapGlobal} and \cref{eq::Pg-bound} is the following:

\begin{assumption}\label{assumption::cts-Lyapunov} Suppose there exist finite subsets $K$ and $A$ such that $K\subseteq A\subseteq S$ and non-negative functions $r,g_1, g_2, h_1,$ and $h_2$ such that $r(x)\geq \lambda(x)$ and
\begin{subequations}
\begin{align}
\sum_{y\in K^c}^{}Q(x,y)g_1(y)\leq -r(x)\label{eq::cts-time-Lyapunov-for-r}
\end{align}
and
\begin{align}
\sum_{y\in K^c}^{}Q(x,y) g_2(y) \leq -1\label{eq::cts-time-Lyapunov-for-e}
\end{align}
\label{eq::cts-time-Lyapunov-global}
\end{subequations}
for $x\in K^c$ and
\begin{align}
\sum_{y\in A^c}^{}Q(x,y)g_{i}(y) \leq h_i(x)\lambda(x)\label{eq42}
\end{align}
for $x\in A$ and $i=1,2$. 
\end{assumption}

Let $Y=(Y_n:n\geq 0)$ be the embedded discrete-time Markov chain having transition matrix $R(x,y)=Q(x,y)/\lambda(x)$ for $x\neq y$ and $R(x,x)=0$ for $x\in S$. Note that \cref{eq::cts-time-Lyapunov-for-r} and \cref{eq42} are equivalent to demanding that
\begin{subequations}
\begin{align}
\sum_{y\in K^c}^{}R(x,y)g_1(y)\leq g_1(x)- \tilde r(x)\label{eq::R-matrix-Lyapunov-for-r}
\end{align}
and
\begin{align}
\sum_{y\in K^c} R(x,y) g_2(y) \leq g_2(x)-\tilde e(x)\label{eq::R-matrix-Lyapunov-for-e}
\end{align}
\label{eq::LyapunovforR}
\end{subequations}
for $x\in K^c$ and
\begin{align}
\sum_{y\in A^c}^{}R(x,y)g_{i}(y)\leq  h_{i}(x) \label{eq::bound-on-Rg}
\end{align}
for $x\in A$ and $i=1,2$, where $\tilde w(x)=w(x)/\lambda(x)$ for a generic non-negative function $w$. Because $\tilde r(x)\geq 1$, \cref{eq::R-matrix-Lyapunov-for-r} ensures that $Y$ is positive recurrent; see {Theorem 11.3.4 in \citet{meynMarkovChainsStochastic2009}, p.266}. This implies that the jump process is non-explosive (see, for example, {Theorem 2.7.1 in \citet{norrisMarkovChains1997}, p.90}). Furthermore, the equilibrium distribution $\nu=(\nu(x):x\in S)$ associated with the jump process is related to the equilibrium distribution $\pi=(\pi(x):x\in S)$ for $(Y_n:n\geq 0)$ via the relation
\begin{align}
\nu(x) = \frac{\pi(x)/\lambda(x)}{\sum_{y\in S}^{}\pi(y)/\lambda(y)}\label{eq::R-stationary-distribution-transformation-to-Q}
\end{align}
for $x\in S$; see, for example,{ Theorem 3.5.1 of \citet{norrisMarkovChains1997}, p.117}. In particular, for a reward function $f$ for which $|f(x)|\leq r(x)$ for $x\in S$, we conclude that
\begin{align}
\sum_{x\in S}^{}\nu(x)f(x) = \frac{\sum_{x\in\mathcal{S}}^{}\pi(x)\tilde f(x)}{\sum_{x\in\mathcal{S}}^{}\pi(x)\tilde e(x)} = \frac{\pi\tilde f}{\pi \tilde e}.\label{eq::cts-time-eq-expectation-in-terms-of-pi-of-R-ratio}
\end{align}
We note that $|\tilde f(x)|\leq \tilde r(x)$. In view of the fact that \cref{eq::LyapunovforR} and \cref{eq::bound-on-Rg} are precisely conditions \cref{eq:LyapGlobal} and \cref{eq::Pg-bound} for the Markov chain $(Y_n:n\geq 0)$ with envelope function $r$ (and $g_1=g_2$), we can now apply the truncation method of \cref{sec::Approximation} to the computation of the ratio $\pi\tilde f/\pi \tilde e$.

We note that \cref{eq::cts-time-eq-expectation-in-terms-of-pi-of-R-ratio} implies that
\begin{align}
\sum_{x\in S}^{}\nu(x)f(x) = \frac{\kappa(\tilde f)}{\kappa(\tilde e)},\label{eq::cts-time-eq-expectation-ratio-representation}
\end{align}
so that the natural truncation approximation for $\sum_{x\in \mathcal{S}}^{}\nu(x)f(x)$ is
\begin{align*}
\tilde\nu_i(f)\overset{\Delta}{=}\frac{\kappa_i(\tilde f)}{\kappa_i(\tilde e)}
\end{align*}
for $i=1,2$. The bounds of \cref{sec::TV-bounds} easily imply the following results.

\begin{theorem} In the presence of \cref{assumption::G-irreducible} and \cref{assumption::cts-Lyapunov}, 
\begin{align*}
\sup_{|f|\leq r}\left|\frac{\kappa_i(\tilde f)}{\kappa_i(\tilde e)}-\sum_{x\in \mathcal{S}}^{}\nu(x)f(x)\right|\leq 2\tilde\epsilon_i(r)
\end{align*}
for $i=1,2$, where
\begin{align*}
\tilde\epsilon_i(r) = \frac{\pi_i\beta_1+\tilde\nu_i(r)\pi_i\beta_2}{\ell_i(\tilde e)} + \Delta_i \left(\frac{\tilde\nu_i(r) \norm{\tilde \kappa_2(\tilde e)}_{\infty} + \norm{\tilde \kappa_1(\tilde r)}_{\infty}}{\ell_i(\tilde e)}\right).
\end{align*}
\end{theorem}

While our Lyapunov condition \cref{assumption::cts-Lyapunov} implies that both the jump process and the embedded discrete-time Markov chain are positive recurrent, our key identity actually holds even when the embedded chain is null recurrent, and our bounds can remain valid even without requiring that $r(x)\geq \lambda(x)$ for $x\in \mathcal{S}$. (This is the part of the Lyapunov condition \cref{assumption::cts-Lyapunov} that guarantees positive recurrence of $(Y_n:n\geq 0)$.)

\section{The Role of the Lyapunov Function}\label{sec::role-of-Lyapunov}
In this section, we briefly discuss the role of the Lyapunov condition \cref{assumption::Lyapunov1}. A consequence of \cref{eq::Lyapunov-for-r} and \cref{eq::Pg-bound} is that for $x\in S$,
\begin{align*}
\sum_{y\in S}^{}P(x,y) g(y)\leq g(x)-r(x)+c
\end{align*}
where $g(x)$ is set equal to 0 for $x\in K$ and
\begin{align*}
c= \max_{x\in K}\left[\sum_{y\in S}^{}P(x,y)g(y)+r(x)\right] =\max_{x\in K} \left[\sum_{y\in A'}^{}P(x,y)g(y) + h(x)+r(x)\right], 
\end{align*}
from which we may conclude that
\begin{align}
\pi r\leq c;\label{eq::MomentBound}
\end{align}
see \citet{glynnBoundingStationaryExpectations2008}. It follows that \cref{eq::Lyapunov-for-r} and \cref{eq::Pg-bound} are a stronger requirement than simply demanding the existence of a ``moment bound'' of the form \cref{eq::MomentBound}. The principal usage of a moment bound, in consideration of state space truncation issues, is via Markov's inequality:
\begin{align*}
\pi\left(\{x\in \mathcal{S}: r(x)\geq s\}\right)\leq \frac{\pi r}{s}.
\end{align*}
In particular, if we choose the truncation set $A$ as $\{x\in \mathcal{S}: r(x)<s\}$, we are guaranteed that
\begin{align}
\pi(A)\geq 1-\frac{\pi r}{s}.\label{eq::TailBound}
\end{align}
By choosing $s$ sufficiently large, we can therefore find a set $A$ for which $\pi(A)\geq 1-\epsilon$, where $\epsilon$ can be chosen as small as desired.

However, the knowledge that $\pi(A)\geq 1-\epsilon$ provides no information on how the chain moves from outside the truncation set $A$ back into $A$. In particular, it tells us nothing about which states $x\in A$ are the ``preferred states'' through which the chain re-enters $A$ from $A^c$. This entrance distribution into $A$ from $A^c$ plays a significant role in determining how $\pi$ assigns probability to the states in $A$. 

Note that when the only problem information available is $B=(P(x,y):x,y\in A)$ and the knowledge of $\epsilon$ for which $\pi(A)\geq 1-\epsilon$, it is generally impossible to construct useful bounds on $(\pi(y):y\in A)$. In particular, consider the transition matrix $P$ given on p.604 of \citet{gibsonAugmentedTruncationsInfinite1987}. It is evident that when $A_n=\{0,...,n\}$, $\pi(A_n^c)<\epsilon/2$ for $n$ sufficiently large. We can now construct two transition matrices $P_1$ and $P_2$ for which $(P_i(x,y):x,y\in A_n)=(P(x,y):x,y\in A_n)$ with $\pi_i(A_n^c)<\epsilon$ for $i=1,2$, specifically $P_1(n+1,0)=1$ and $P_2(n+1,n)=1$. The matrix $P_1$ describes the dynamics of ``first-state'' augmentation while $P_2$ describes the dynamics under ``last state'' augmentation. Since first state augmentation converges and last state augmentation fails to converge for this example, it is clear that bounds based only on such problem data can not converge in general.

Consequently, some information related to how the chain transitions from $A^c$ back into $A$ is needed. The Lyapunov conditions \cref{eq:LyapGlobal} and \cref{eq::Pg-bound} provide such information that can be used to build error bounds. 

Finally, suppose that $S=\mathbb Z_+^d$ and that one develops a moment bound on $\pi r$ for some function $r$ that grows polynomially in $x=(x_1,...,x_d)$. A natural choice for the truncation set $A$ will then be a set of the form $\{x\in S: \norm{x}_1\leq m\}$ for some suitable choice of $m$. In this case, \eqref{eq::TailBound} yields a bound that decays polynomially in $m$.

However, it can often be the case that the Markov chain has much better tail behavior than that obtained from the moment bound. For example, suppose that $\pi$ actually has exponentially decaying tails, but that $r$ has been chosen to grow polynomially. In this setting, any algorithm that exploits only the tail bound \eqref{eq::TailBound} can only obtain a polynomially decaying error bound.

On the other hand, our error bounds can often still recover exponential decay, even in the presence of poor choices for $g_1$ and $g_2$. For example, suppose that (as is often the case for such models) one uses Lyapunov functions $g_1$ and $g_2$ that grow polynomially in $x$ to verify \eqref{eq:LyapGlobal} for $r$. In such a setting, we see that the Lyapunov error bounds \eqref{eq::singleton-bounds} then should decay exponentially, since
\begin{align}
E_xg_i(X_T) I(T<T_K) &\leq \frac{1}{\pi(x)} E_{\pi_K}g_i(X_T) I(T<T_K)\nonumber\\
&\leq \frac{1}{\pi(x)}E_{\pi_K}\sum_{j=0}^{T_K-1} g_i(X_j)I(X_j\in A^c)\nonumber\\
&=\frac{1}{\pi(x)}E_{\pi_K}T_K\cdot \sum_{y\in A^c}\pi(y)g_i(y)\label{eq::ExponentialDecayofErrorBound}
\end{align}
for $x\in K$. In particular, if $\sum_{\norm{x}_{1}>m}^{}\pi(x)$ decays exponentially in $m$, \eqref{eq::ExponentialDecayofErrorBound} goes to zero exponentially. Hence, even with polynomial choices for $g_1$ and $g_2$, our bounding approach can still recover exponential decay in our computed bounds.

\section{Extension to Reducible Markov Chains\label{sec::Reducible}}
We now briefly discuss the extension of our approximations to the setting of reducible Markov chains. We consider only the case in which $S$ is the union of $s$ disjoint positive recurrent closed communicating classes $\mathcal{C}_1, \mathcal{C}_2,...,\mathcal{C}_s$, plus (possibly) an additional set $\mathcal{T}$ of transient states from which eventual absorption into one of the $\mathcal{C}_i$'s is certain. (In particular, we do not consider the possibility of infinitely many closed communicating classes, or closed communicating classes that are transient.) If we are aware of the reducible structure of our chain, we can apply the method previously discussed in this paper to the computation of approximations to each of the individual equilibrium distributions $\pi(\mathcal{C}_1,\cdot),...,\pi(\mathcal{C}_s, \cdot)$ associated with the $s$ closed communicating classes.

If we are unaware of the reducible structure of the chain and apply the ideas of this paper, we note that $K$ would typically intersect a number of the $\mathcal{C}_i$'s. Define $K_i=K\cap \mathcal{C}_i$ for $i=1,...,s$, and without loss of generality, assume $K_i=\emptyset$ if and only if $i>s'$. The matrix $G$ would (for $A=A_n$ large enough, with $A_n\nearrow S$) eventually be guaranteed to have block diagonal structure, corresponding to the blocks $K_1,...,K_{s'}$. The states in $K-(K_1\cup...\cup K_{s'})$ would, of course, be transient states, and each of the diagonal blocks associated with the $K_i$'s would be irreducible.

When $s'\geq 2$, the block diagonal structure in $G$ would be a diagnostic test for the reducibility of the original chain. If $G_i$ is the irreducible principle sub-matrix of $G$ corresponding to $K_i$, row normalization would lead to an irreducible stochastic matrix $P_2(K_i)=(P_2(K_i,x,y):x,y\in K_i)$, which would approximate the transition matrix $P_{K_i}$ of the censored chain on $K_i$. If $\pi_2(K_i)=(\pi_2(K_i, x):x\in K_i)$ is the (unique) equilibrium distribution associated with $P_2(K_i)$, then we may approximate the equilibrium reward associated with $\mathcal{C}_i$'s equilibrium distribution $\pi(\mathcal{C}_i)=(\pi(\mathcal{C}_i,x):x\in \mathcal{C}_i)$ via
\begin{align*}
\tilde \pi(\mathcal{C}_i, r) = \frac{\sum_{x\in K_i}^{}\pi_2(K_i,x)\utilde \kappa(x,r)}{\sum_{x\in K_i}^{}\pi_2(K_i,x)\utilde \kappa(x,e)},
\end{align*}
where $\utilde \kappa(r)$ and $\utilde\kappa(e)$ can (if one wished) be computed exactly as discussed in \cref{sec::Approximation} of this paper. In other words, $\utilde\kappa(r)$ and $\utilde\kappa(e)$ can be computed via calculations that are framed in terms of $K$ and $A$, and that ignore the reducible structure of $P_{12}$ and $P_{22}$. Of course, these computations would typically run faster if one took explicit advantage of the block structure that is present. Furthermore, in the presence of \cref{eq:LyapGlobal} and \cref{eq::Pg-bound}, the Lyapunov bounds carry over to this reducible setting, so that with $\utilde\kappa(x,f),\tilde\kappa_{1}(x,r),$ and $\tilde\kappa_{2}(x,e)$ defined exactly as in \cref{sec::EqExpBounds}, the bounds
\begin{align*}
\utilde\kappa(x,r) &\leq \kappa(x,r)\leq \tilde \kappa_1(x,r)\\
\utilde\kappa(x,e) &\leq \kappa(x,e)\leq \tilde \kappa_2(x,e)\\
\end{align*}
continue to hold in the reducible context. 

In general, we can not expect to be able to compute bounds on equilibrium moments for reducible chains unless $s'=s$. If $s'=s$, then any stationary distribution $\pi$ for the reducible model satisfies
\begin{align*}
\pi r \leq \max_{1\leq i\leq s} \frac{\max_{\phi\in \mathcal{P}_i}\sum_{x\in K_i}^{}\phi(x)\tilde\kappa_1(x,r)}{\min_{\phi\in \mathcal{P}_i}\phi(x)\utilde\kappa(x,e)},
\end{align*}
where $\mathcal{P}_i=\{\phi\in\mathcal{P}: \phi(x)=0 \text{ for } x\not\in K_i\}$. Similarly, 
\begin{align*}
\pi r\geq\min_{1\leq i\leq s}\frac{\min_{\phi\in \mathcal{P}_i}\sum_{x\in K_i}^{}\phi(x)\utilde\kappa(x,r)}{\max_{\phi\in \mathcal{P}_i}\sum_{x\in K_i}^{}\phi(x)\tilde\kappa_2(x,e)}.
\end{align*}
Hence, even if our algorithm of \cref{sec::Approximation} is applied to a reducible chain, the bounds remain valid and useful.

\section{Computational Experiments}\label{sec::NumericalExperiments}
In this section, we compare our proposed algorithms with two competing algorithms that also come equipped with error bounds and convergence guarantees. In particular, we will compare our algorithms to Repeated Truncation and Augmentation (RTA) (see \citet{courtoisBoundsPositiveEigenvectors1984,courtoisPolyhedraPerronFrobeniusEigenvectors1985,courtoisComputableBoundsConditional1986}, \citet{semalRefinableBoundsLarge1995} (discrete time) \citet{dayarBoundingEquilibriumDistribution2011}, \citet{spielerNumericalAnalysisLongrun2014} (continuous time), and the discussion and references in the review \citet{kuntzStationaryDistributionsContinuoustime2021}) and the Linear Programming Outer Approximation (LPOA) (see \citet{kuntzBoundingStationaryDistributions2019,kuntzApproximationsCountablyInfinite2021} and the discussion in the review \citet{kuntzStationaryDistributionsContinuoustime2021}). Specifically, we will compare algorithmic performance on the $G/M/1$ model that arises in queueing theory, and on the ``toggle switch'' biological model recently utilized by \citet{kuntzStationaryDistributionsContinuoustime2021} to study equilibrium distribution computation.

The $G/M/1$ queue is a single-server queue with independent and identically distributed (iid) interarrival times drawn from a distribution $G$, and iid service times drawn from an exponential distribution with rate parameter $\mu$. If $X_n$ corresponds to the number of customers in the queue just before the $n$'th arrival, then $X_n$ is a Markov chain on $S=\mathbb Z_+\overset{\Delta}{=}\{0,1,2,...\}$ with one-step probability transition matrix $P=(P(x,y):x,y\in \mathbb Z_+)$ having entries given by $P(x,y)=\beta_{x+1-y}$ for $1\leq y\leq x+1$ and $P(x,y)=0$ for $y>x+1$, where
\begin{align*}
\beta_i = \int_{0}^{\infty}e^{-\mu t}\frac{(\mu t)^i}{i!}G(dt)
\end{align*}
for $i\geq 0$ and $\beta_i=0$ for $i<0$; see \citet{asmussenAppliedProbabilityQueues2008} for details about this model. Specifically, we assume that $G(dt)=I(0\leq t\leq 2.01)dt/(2.01)$, and $\mu=1$. With this problem data, the $G/M/1$ queue is positive recurrent and has equilibrium distribution $\pi=(\pi(x):x\in \mathbb Z_+)$ given by 
\begin{align}
\pi(x)=(1-\theta)\theta^x\label{eq::GM1Geometric}
\end{align}
where $\theta=1-\xi/\mu$ and $\xi$ is the positive solution of 
\begin{align}
1 = \frac{\mu}{\mu-\xi}\int_{0}^{\infty}e^{-\xi x}G(dx)\label{eq::GM1RootProblem};
\end{align}
see p.279 of \citet{asmussenAppliedProbabilityQueues2008}.

To construct suitable Lyapunov functions for this example, suppose that $r(x)=x$ for $x\geq 0$ and note that
\begin{align*}
(Pg)(x) = Eg([x+1-V]^+),
\end{align*}
where $V$ has probability mass function $(\beta_j:j\geq 0)$. Put $g_1(x)=c_1x^2$ (with $c_1>0$) and observe that for any $K\subseteq S$, 
\begin{align*}
\sum_{y\in K^c}^{}P(x,y)g_1(y) &\leq \sum_{y\in S}^{}P(x,y)g_1(y)\\
&\leq Ec_1(x+1-V)^2\\
&= g_1(x) + 2c_1xE(1-V) + c_1E(1-V)^2\\
&= g_1(x) - r(x) - (2c_1 -200)(EV-1)x + c_1E(1-V)^2.
\end{align*}
Put $c_1=300$ and $n_1 = \lceil c_1E(1-V)^2/((2c_1-200)(EV_1-1))\rceil$, so that $(Pg_1)(x)-g_1(x)+r(x)$ is guaranteed to be non-positive for $x\geq n_1$. For $g_2$, we select $g_2(x)=c_2x$ (with $c_2>0$) and again note that for any $K\subseteq S$,
\begin{align*}
\sum_{y\in K^c}^{}P(x,y)g_2(y) &\leq \sum_{y\in S}^{}P(x,y)g_2(y)\\
&= c_2 E[(x+1-V)+ (V-x-1)I(V>x)]\\
&\leq g_2(x)-c_2E(V-1) + c_2 EVI(V>x)\\
&\leq g_2(x)-1 + (200-c_2)(EV-1)+\frac{c_2EV^3}{x^2}.
\end{align*}
We select $c_2=300$ and put $n_2 =\lceil (c_2 EV^3/((c_2-200)(EV-1)))^{1/2}\rceil$, so that $(Pg_2)(x)-g_2(x)+1\leq 0$ for $x\geq n_2$. Since $P$ has a finite number of non-zero entries in each row, we can numerically compute $(Pg_i)(x)$ $(i=1,2)$ for any $x$ that we wish. In view of this numerical consideration, we can therefore set $K=\{0,1,2,...,k^*\}$ where $k^*=\max\{x\in\mathbb Z_+: (Pg_1)(x)-g_1(x) + r(x)> 0 \text{ or } (Pg_2)(x)-g_2(x)+1> 0\}$ (which is guaranteed to be less than or equal to $\max(n_1,n_2)$). With our specified problem data, $n_1=202, n_2=66,$ and $k^*=201$. 

In our computations, we study the behavior of our algorithm as a function of the size of the truncation set $A$. Regardless of the size of $A$, the sparsity structure of $P$ allows us to compute $h_i(x)=\sum_{y\in A^c}^{}P(x,y)g_i(y)$ numerically for $x\in A$ and $i=1,2$. Because of this, and because $\pi$ is a geometric distribution for this example, so that all of its moments are finite and in particular $\pi g_i<\infty$, we conclude that our bounds are convergent.

The RTA algorithm and the LPOA algorithm require an upper bound on $\pi w$ for some norm-like function (or, equivalently, column vector) $w$. To obtain the required bound, we seek a non-negative (Lyapunov) function $g_3$ for which there exists $c\in \mathbb R_+$ for which
\begin{align}
(P g_3)(x)\leq  g_3(x)-w(x)+c\label{eq::FullLyapunov}
\end{align}
for all $x\in S$. \citet{glynnBoundingStationaryExpectations2008} show that \cref{eq::FullLyapunov} implies that $\pi w\leq c$. For the $G/M/1$ queue, we put $g_3(x)=c_3x^4$ and $w(x)=x^3$ and observe that
\begin{align*}
(P g_3)(x) &\leq c_3E(x+1-V)^4\\
&= g_3(x)-x^3+(4c_3(1-EV)+1)x^3 + 6c_3E(1-V)^2x^2 + \\
&\qquad \qquad 4c_3 E(1-V)^3x +c_3 E(1-V)^4.
\end{align*}
It can be shown that if $(4c_3(1-EV)+1)<0$ and $n_3 \geq \lceil 1+\max\{|\frac{a_2}{a_3}|, |\frac{a_1}{a_3}|, |\frac{a_0}{a_3}|\}\rceil $ with $a_0=c_3 E(1-V)^4$, $a_1 = 4c_3 E(1-V)^3$, $a_2 = 6c_3E(1-V)^2, a_3 =  4c_3(1-EV)+1$ we are guaranteed $(P g_3)(x)\leq g_3(x)-w(x)$ for $x\geq n_3$. Hence, if we set
\begin{align*}
c=\max\{(P g_3)(x)- g_3(x)+w(x):0\leq x\leq n_3\}, 
\end{align*}
\cref{eq::FullLyapunov} holds, so $\pi w\leq c$. For our problem data, we set $c_3=300$ and find $n_3=1803$ and $c \leq 8.3 \times 10^{7}$. With the moment bound in hand, we can use (4.35), (4.36), (4.37), (4.38), and (4.40) in \citet{kuntzStationaryDistributionsContinuoustime2021} to produce the relevant bounds for RTA and LPOA.

We turn next to the toggle switch model that describes the random evolution of the total number of protein molecules of two different types over time. We study the symmetric case where the rate of synthesis of type 1 (type 2) proteins is $\lambda/(1+x_2)$ ($\lambda/(1+x_1)$), where $x_i$ is the number of type $i$ proteins ($i=1,2$). Also, the rate at which each protein molecule decays is $\mu$, regardless of type. This leads to a rate matrix $Q=(Q(x,y):x,y\in \mathbb Z_+^2)$ for which $Qf$ is given by 
\begin{align*}
(Qf)(x_1,x_2) &= \frac{\lambda}{1+x_2}\left(f(x_1+1,x_2) - f(x_1,x_2)\right)\ +\frac{\lambda}{1+x_1}(f(x_1,x_2+1)-f(x_1,x_2))\\
&\qquad+ \mu x_1(f(x_1-1,x_2)-f(x_1,x_2)) + \mu x_2(f(x_1,x_2-1)-f(x_1,x_2))
\end{align*}
for $(x_1,x_2)\in\mathbb Z_+^2$ with $x_1,x_2\geq 1$ and $f:\mathbb Z_+^2\rightarrow \mathbb R_+$. We specifically study two variants of this model, one with $\lambda=20$ and $\mu=1$, which we denote as TS(20,1) and the other with $\lambda=90$ and $\mu=1$, which we denote as TS(90,1). TS(20,1) was analyzed by \citet{kuntzStationaryDistributionsContinuoustime2021} in their numerical study.

In contrast to the $G/M/1$ queue, there is no closed form for $\pi$ for these toggle switch models. However, we can derive an approximation $(x^*,x^*)$ to the mode(s) of $\pi$ by noting that the birth and death rates for type $i$ proteins are approximately equal at the integer point(s) closest to $(x_1^*, x_2^*)$ at which $\lambda=(1+x_2^*)\mu x_1^*=(1+x_1^*)\mu x_2^*$. It follows that $x_1^*=x_2^*=x^*$ and $\mu {x^*}^2+\mu x^*-\lambda=0$. We find that the integer point closest to $(x^*,x^*)$ is $(4,4)$ for TS(20,1) and $(9,9)$ for TS(90,1) (In fact, $x^*$ is integral for these models.). For this model, our function $r$ is the total number of proteins, namely $r(x)=x_1+x_2$. We take $g_1(x)=(x_1-x^*)^2+(x_2-x^*)^2$ and find
\begin{align*}
(Qg_1)(x) &= \frac{\lambda}{1+x_2}\left(2(x_1-x^*)+1\right) + \frac{\lambda}{1+x_1}(2(x_2-x^*)+1) \\
&\qquad+ \mu x_1(-2(x_1-x^*)+1) + \mu x_2(-2(x_2-x^*)+1)\\
&\leq  2\lambda + 2\lambda (x_1+x_2) + 2\mu(2x^*+1)(x_1+x_2) - 2\mu (x_1^2+x_2^2)\\
&= -r(x) +2\lambda + (1 + 2\lambda +2\mu(2x^*+1))(x_1+x_2)- 2\mu (x_1^2+x_2^2)
\end{align*}
when $x^*\geq 1/2$, as is the case in our models. A calculus argument shows that when $x_1+x_2\geq n_1= \lceil(c_1 + \sqrt{c_1^2 + 4c_2c_0/2})/c_2\rceil$ with $c_0 = 2\lambda, c_1 = (1 + 2\lambda +2\mu(2x^*+1))$, and $c_2 = 2\mu$, we have $(Qg_1)(x)\leq -r(x)$. We pick $g_2(x)=|x_1-x^*|+|x_2-x^*|$ and we find,
\begin{align*}
(Qg_2)(x) &= \frac{\lambda}{1+x_2}\left(|x_1+1-x^*|-|x_1-x^*|\right)+ \frac{\lambda}{1+x_1}\left(|x_2+1-x^*|-|x_2-x^*|\right)\\
&\qquad + \mu x_1 (|x_1-1-x^*|-|x_1-x^*|)+ \mu x_2 (|x_2-1-x^*|-|x_2-x^*|)\\
&\leq 2\lambda- \mu x_1 I(x_1> x^*)- \mu x_2 I(x_2> x^*)+ 2\mu x^*\\
&\leq 2 \lambda - \mu(x_1+x_2) + 4\mu x^*\\
&\leq -1
\end{align*}
when $x_1+x_2\geq n_2 =\lceil(2 \lambda+4\mu x^*+1)/\mu\rceil$. For TS(20,1) (resp. TS(90,1)) we have $n_1=60$ and $n_2=57$ ($n_1=220$, $n_2=217$). We can now put $K=\{(x_1,x_2)\in Z_+^2:x_1+x_2<\max(n_1,n_2), (Qg_1)(x_1,x_2)>-r(x_1,x_2) \text{ or }(Qg_2)(x_1,x_2)>-1)\}$. (Observe that $(Qg)(x)$ is numerically computable for each $x$, since $Q$ has finitely many non-zeros in each row.) The resulting set $K$ is illustrated in \cref{fig::K_sets}. 

\begin{includefigures}
 \begin{figure}[H]
\subfloat[]{\label{sf_K_union}
\includegraphics[width=60mm]{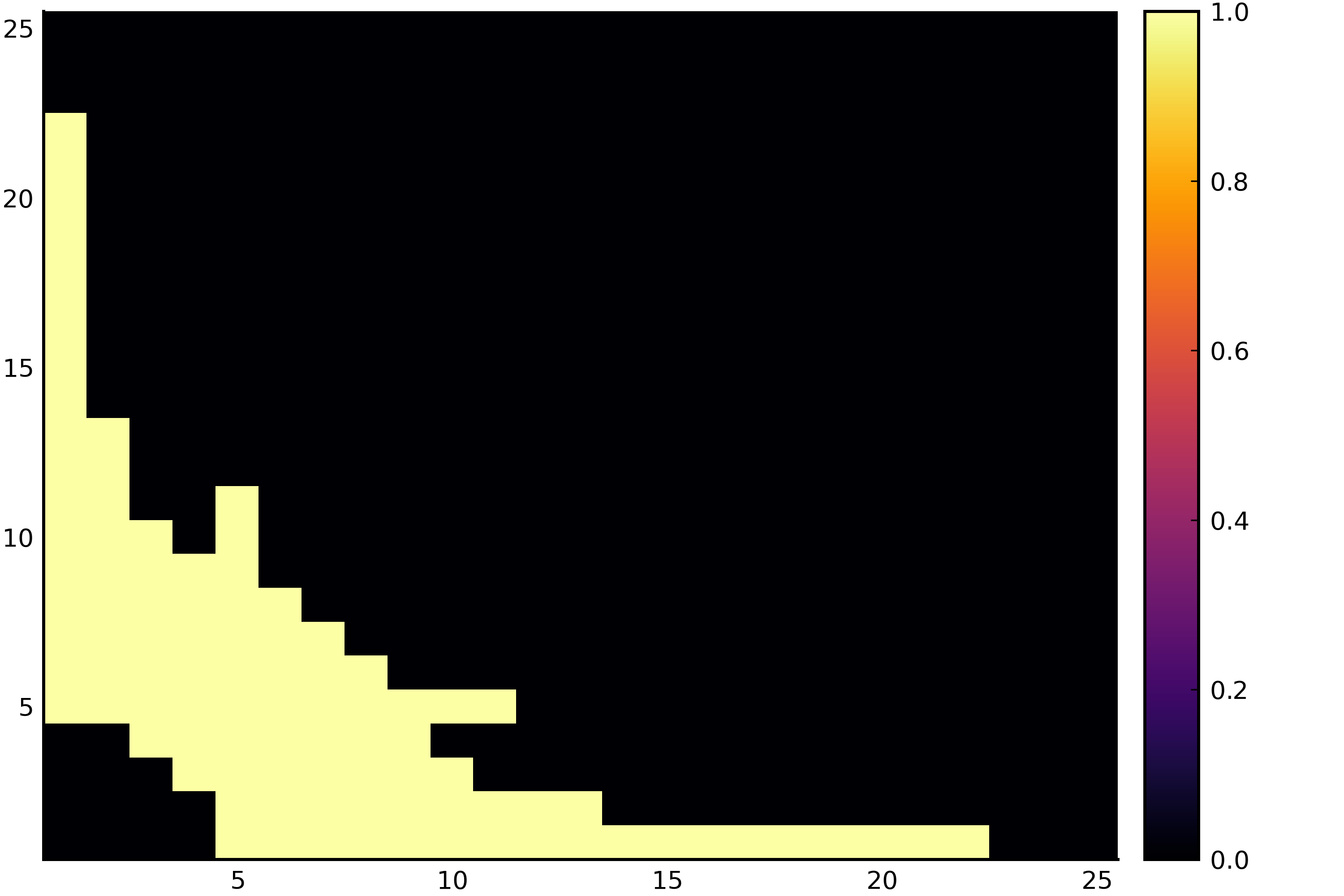}
}
\subfloat[]{\label{sf_K_tandem}
\includegraphics[width=60mm]{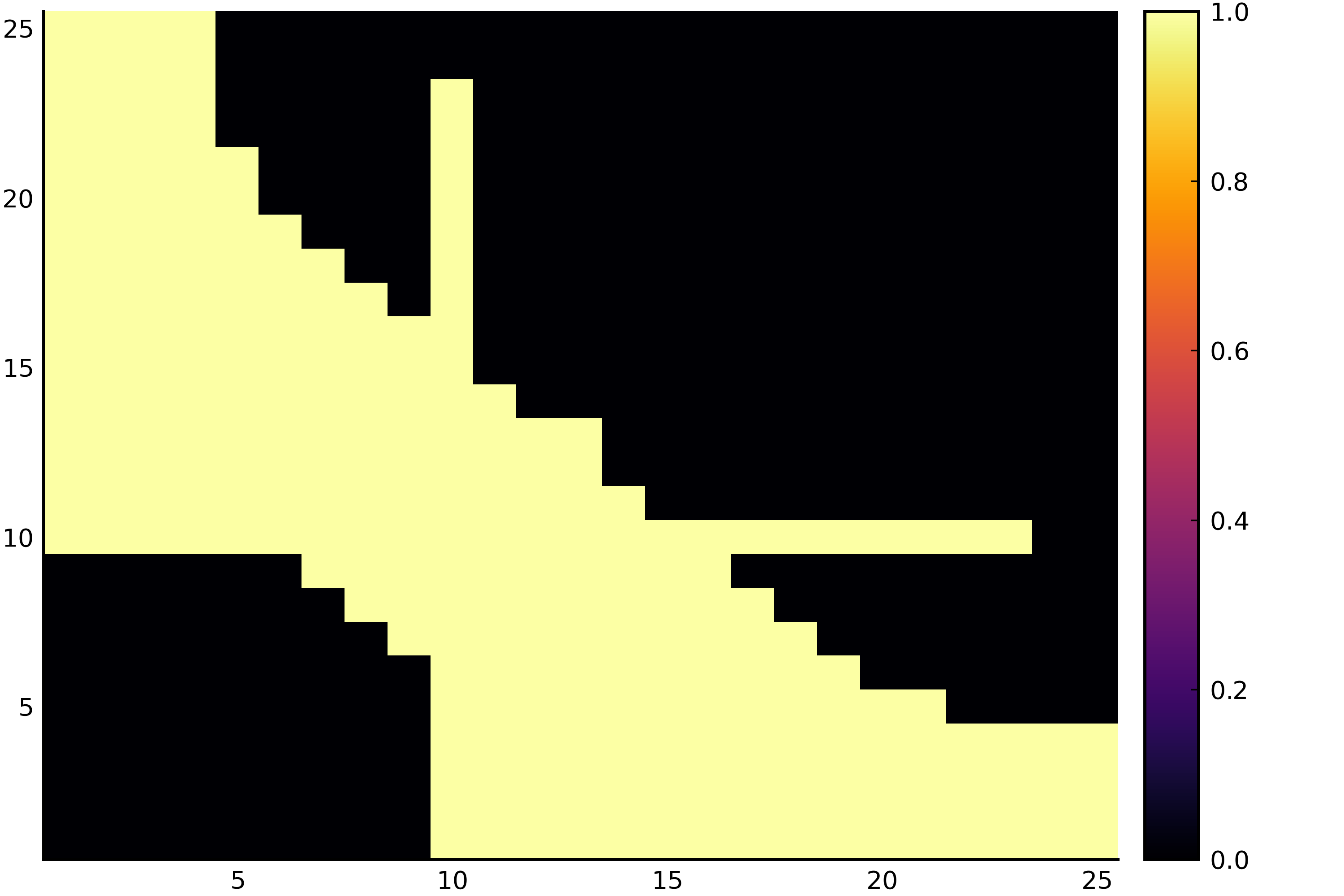}
}\\
\caption{The set $K$ for TS(20,1) (a) and TS(90,1) (b).\label{fig::K_sets}}
\end{figure}
\end{includefigures}

As with the $G/M/1$ queue, we can compute $h_i(x)=\sum_{y\in A^c} P(x,y)g_i(y)$ exactly for TS(20,1) and TS(90,1). Furthermore, the analysis of \citet{kuntzBoundingStationaryDistributions2019} Appendix B carries over to these models, so that we can guarantee that $(x_1+x_2)^p$ is $\pi$-integrable for all $p$. This implies that $x_1^p+x_2^p$ is $\pi$-integrable for all $p$, so that we conclude $|x_1-x^*|^p+|x_2-x^*|^p$ is $\pi$-integrable for all $p$. We conclude that $\pi g_i<\infty$ and therefore that our bounds are convergent.

Turning next to the bound on $\pi w$ (for some norm-like function $w$ that goes to infinity at least as fast as $r$), we can use the bound computed by \citet{kuntzStationaryDistributionsContinuoustime2021} via semi-definite programming for TS(20,1). Specifically, their bound asserts that $\pi w\leq 1.8\times 10^{7}$ when $w(x_1,x_2)=(x_1+x_2)^6$. For the moment bound for TS(90,1), we again construct (as in the $G/M/1$ setting) a Lyapunov function for bounding $\pi w$. Specifically, we use $g_3(x)=\alpha g_1(x)$ with $g_1(x)=(x_1-x^*)^2+(x_2-x^*)^2$ as above for $w(x)=(x_1+x_2)^2$. Another calculus argument establishes that for $\alpha>1$, $x_1+x_2\geq n_3=\lceil (\alpha c_1+\sqrt{\alpha^2c_1^2+4(\alpha c_2/2-c_3)\alpha c_0})/(2(\alpha c_2/2-c_3))\rceil$ with $c_3=1$ and $\alpha c_2/2-c_3\geq 0$, we have $(Qg_3)(x)\leq -w(x)$. We therefore have the moment bound
\begin{align*}
\pi w \leq \max_{x:x_1+x_2\leq n_3} \left((Qg_3)(x)+w(x\right)).
\end{align*}
We fix $\alpha=4$ and find that $n_3 = 293$ and achieve the bound $\pi w\leq 16.4\times 10^3$.

We now turn to the computational implementations of the algorithms. First we describe a reformulation of the bounds of \cref{sec::EqExpBounds} that we found was more numerically stable. Then we describe implementation details for our computational experiments in the Julia programming language (version 1.6.2).

Note that the optimization problems in \cref{eq::CS-lower-and-upper-bounds} do not require the full matrix inversion of $I-G$, because for a column vector $q=(q(x):x\in K)$ we have
\begin{align}
\tau_x q = \frac{\xi_1(x)}{\xi_2(x)},\label{eq::CS-bounds-in-terms-of-general-q}
\end{align}
where $\xi_1=(\xi_1(x):x\in K)$ is a column vector satisfying $(I-G)\xi_1=q$ and $\xi_2=(\xi_2(x):x\in K)$ is a column vector satisfying $(I-G)\xi_2=e_1$. When $A$ is large relative to $K$, we expect $G$ to be close to stochastic, so that $I-G$ is then ill-conditioned. We now show how to reformulate the two linear systems arising in \cref{eq::CS-bounds-in-terms-of-general-q} into two linear systems involving a $(|K|-1)\times (|K|-1)$ principal submatrix $\widehat G$ of $G$ obtained by deleting the row and column of $G$ associated with a state $z\in K$. A good choice of $z$ may lead to a better conditioned matrix $I-\widehat G$. We found this to be computationally helpful in our examples below. 

Put $\widehat K=K-\{z\}$ and note that for a generic vector $q$ and $x\in K$,
\begin{align}
((I-G)^{-1}q)(x) &= \sum_{j=0}^{\infty}(G^jq)(x)\nonumber\\
&=E_x \sum_{j=0}^{\widehat T-1}q(X_{T_K(j)})\label{eq::numerical-stable-ImGq}
\end{align}
where $T_K(j)$ is the time at which $X$ returns to $K$ for the $j$'th time, and $\widehat T=\inf\{j\geq 1: T<T_K(j)\}$. For $z\in K$, set $\widehat \tau(z)=\inf\{j\geq 1: X_{T_K(j)}=z\}$, and observe that
\begin{align}
E_x \sum_{j=0}^{\widehat T-1}q(X_{T_k(j)}) &=  E_x \sum_{j=0}^{(\widehat T\land \widehat \tau(z))-1}q(X_{T_K(j)}) + P_x(\widehat \tau(z)<\widehat T)E_z\sum_{j=0}^{\widehat T-1}q(X_{T_K(j)})\label{eq::numerical-stable-sum-q-until-hatT}
\end{align}
for $x\in K$. Hence, for $x\in \widehat K$,
\begin{align}
((I-G)^{-1}q)(x) = \left((I-\widehat G)^{-1}\widehat q\right)(x) + ((I-\widehat G)^{-1}\chi)(x)((I-G)^{-1}q)(z),\label{eq::numerical-stable-3}
\end{align}
where $\widehat q=(\widehat q(x):x\in \widehat K)$ with $\widehat q(x)=q(x)$ for $x\in \widehat K$, and $\chi=(\chi(x):x\in \widehat K)$ is defined by $\chi(x)=G(x,z)$ for $x\in \widehat K$.

Similarly, 
\begin{align*}
((I-G)^{-1}q)(z) &= q(z) + \sum_{x\in \widehat K}^{}G(z,x)((I-\widehat G)^{-1}\widehat q)(x) \\
&\qquad + \sum_{x\in \widehat K}^{}G(z,x)((I-\widehat G)^{-1}\chi)(x)((I-G)^{-1}q)(z),
\end{align*}
from which we may conclude that
\begin{align}
((I-G)^{-1}q)(z) = \frac{q(z) + \sum_{x\in \widehat K}^{}G(z,x)((I-\widehat G)^{-1}\widehat q)(x)}{1-\sum_{x\in \widehat K}^{}G(z,x)((I-\widehat G)^{-1}\chi)(x)}\label{eq::numerical-stable-4}
\end{align}
Relations \eqref{eq::numerical-stable-ImGq}-\eqref{eq::numerical-stable-4} allow us to compute $(I-G)^{-1}q$ in terms of $(I-\widehat G)^{-1}\widehat q$ and $(I-\widehat G)^{-1}\chi$. Since $\widehat G$ converges to a strictly substochastic matrix as $A\nearrow S$, $I-\widehat G$ is much better conditioned than is $(I-G)$ when $A$ is large. Of course, the vector $(I-G)^{-1}e_1$ can be similarly computed in terms of linear systems involving the coefficient matrix $(I-\widehat G)$. It follows that
\begin{align*}
\tau_zq = \frac{((I-G)^{-1}q)(z)}{((I-G)^{-1}e_1)(z)} = \frac{q(z)+ \sum_{x\in \widehat K}^{}G(z,x)((I-G)^{-1}\widehat q)(x)}{1+\sum_{x\in \widehat K}^{}G(z,x)((I-\widehat G)\widehat e_1)(x)},
\end{align*}
where $\widehat e_1=(\widehat e_1(x):x\in \widehat K)$ has entries $\widehat e_1(x)=1$ for $x\in \widehat K$. Similar identities for $\tau_x q$ for $x\in \widehat K$ can be easily derived, thereby providing expressions that involve the better conditioned $I-\widehat G$. In our experiments, we choose $z$ to be the $x\in K$ that maximizes $\sum_{y\in K}^{}G(x,y)$.

These experiments were conducted in the Julia programming language (version 1.6.2) (see \citet{DunningHuchetteLubin2017}). All linear systems that arise in the computation of the approximation of \cref{sec::Approximation} and bounds of \cref{sec::EqExpBounds,sec::TV-bounds} were solved using Julia's backslash operator, which defaults to using a dense (sparse) $LU$ factorization for dense (sparse) square matrices that do not exhibit further structure like symmetry or triangularity. On the way to computing $G$, we compute $H=P_{21}(I-P_{22})^{-1}$. When $P$ is sparse, the $|K|$ linear systems that arise in the computation of $H$ have sparse coefficient matrices and sparse right hand sides. Because Julia version 1.6.2 does not support solves with a sparse right hand side for sparse matrices, the matrix $H$ (and therefore $G$) is computed as a dense matrix. For the models and choices of $K$ in this paper, $G$ may have many zeros, so we then convert $G$ to a sparse array. Computing $G$ as a dense matrix may not always be advisable, but we achieve good performance with this implementation. We use row normalization to stochasticize $G$ for all of our experiments. (Initial experiments suggested that the approximation based on Perron-Frobenius normalization underperforms the approximation based on row normalization on these models.) We find the stationary distribution associated with $P_2$ using the ``Remove an Equation'' approach of Section 2.3.1.3 in \citet{stewartIntroductionNumericalSolution1994}.

RTA and LPOA are described in detail in \citet{kuntzStationaryDistributionsContinuoustime2021}. For RTA, as above, we again use the ``Remove an Equation'' approach of Section 2.3.1.3 in \citet{stewartIntroductionNumericalSolution1994} to compute the stationary distributions for each of the augmentations. For LPOA, we use JuMP with Mosek, and this time set \texttt{MSK\_IPAR\_OPTIMIZER} equal to 4,``INTPNT'' (which directs the Mosek LP solver to use its interior point method), and optimize using the the variables $v(x)=\pi(x)\lambda(x)$. We found that these settings typically performed better than other settings on this class of problems. When computing the approximation for LPOA, we follow the suggestion of \citet{kuntzStationaryDistributionsContinuoustime2021} and maximize the total mass of the approximating measure (see (4.53) in \citet{kuntzStationaryDistributionsContinuoustime2021}).

We turn next to a brief discussion of the complexity of the algorithms. RTA requires solving as many $|A|\times |A|$ linear systems as there are states in the ``in-boundary'', given by $\mathcal{B}_i(A)\overset{\Delta}{=}\{x\in A: \text{there exists } y\in A^c \text{ with }P(y,x)>0\}$. Meanwhile, LPOA requires solving two LPs having $|A|$ decision variables. All of our methods require first computing $G$, which requires solving $|K|$ $|A'|\times |A'|$ linear systems. The numerically stable implementation of the bounds of \cref{sec::EqExpBounds} require solving 6 linear systems with a $|K|-1\times |K|-1$ matrix. Meanwhile, our minorization TV bounds require inverting a $|K|\times |K|$ matrix and computing $\binom{|K|}{2}$ total variation distances between distributions of size $|K|$. (Note that the $\binom{|K|}{2}$ above can be reduced to $|K|-1$ at the cost of computing a bound that is at most twice the original total variation bound, by noting that the triangle inequality implies that for any $w\in K$
\begin{align*}
\max_{x,y\in K} \sum_{z\in K}^{}|\tau_x(z)-\tau_y(z)|\leq 2 \max_{y\in K}\sum_{z\in K}^{}|\tau_w(z)-\tau_y(z)|.
\end{align*}
We did not find using this trick necessary to achieve state of the art performance in our experiments.). Our perturbation TV bounds require inverting the $|K|\times |K|$ fundamental matrix. Hence, when $|A|>>|K|^2$ (as occurs when $A_n\nearrow S$ with $K$ fixed), we expect our equilibrium expectation bounds to be computable in less time than is the case with LPOA. (We note that this conclusion relies on the claim that, at least with practical algorithms, one would expect solving a constant number of linear systems of size $|A'|\approx |A|$ to scale more favorably with $|A|$ than solving an LP with $|A|$ decision variables.) For models for which the in-boundary grows to infinity when $A_n\nearrow S$ (for example, the $G/M/1$ queue and toggle switch models), we can expect our bounds and approximation to be computed in less time than RTA eventually. More generally, note that because the ``boundary'' of $A$ typically has an increasing number of elements as $A_n\nearrow S$ when $S=\mathbb Z_+^d$ for $d\geq 2$, we can expect the runtime of RTA to typically degrade more quickly relative to our methods for models in two or more dimensions. 

To compute our approximation errors for the various algorithms, we wish to numerically calculate the ``exact'' $\pi$'s for our models. We call these numerically computed ``exact'' $\pi$'s our reference solutions. While we have an analytical closed form for $\pi$ as associated with the $G/M/1$ queue (see \cref{eq::GM1Geometric}), initial experiments suggested that the root $\tau$ appearing in \cref{eq::GM1RootProblem} may be difficult to compute to high precision. As a result, we use our approximation to $\pi$ with $K=\{0,1,...,4\}$, and $A=\{0,1,...,10000\}$ as our reference solution. In particular, the bounds of \cref{sec::TV-bounds} yield a total variation error bound of less than $10^{-12}$ for the reference solutions with this choice of $K$ and $A$. We also compute a reference approximation to $\pi r$ using $K=\{0,1,2,...,201\}$ and $A=\{0,1,...,10000\}$. In this case, the bounds of \cref{sec::TV-bounds} provide an $r$-weighted total variation guarantee of less than $10^{-6}$.

For TS(20,1) and TS(90,1) we use $K$ as described above and use the same $A=\{x\in \mathbb Z_+^2: x_1+x_2\leq 200\}$ for both models. For TS(20,1), the bounds of \cref{sec::TV-bounds} give a total variation error guarantee of less than $10^{-13}$ and an $r$-weighted total variation guarantee of less than $10^{-11}$ for our approximation. For TS(90,1), we achieve a total variation error guarantee of less than $10^{-12}$ and an $r$-weighted total variation guarantee of less than $10^{-10}$. We show the reference solutions restricted to $\{x\in \mathbb Z_+^2: x_1+x_2\leq 25\}$ in \cref{fig_reference_solutions}. 

\begin{includefigures}
 \begin{figure}[H]
\subfloat[]{\label{sf_KuntzTS_Reference}
\includegraphics[width=60mm]{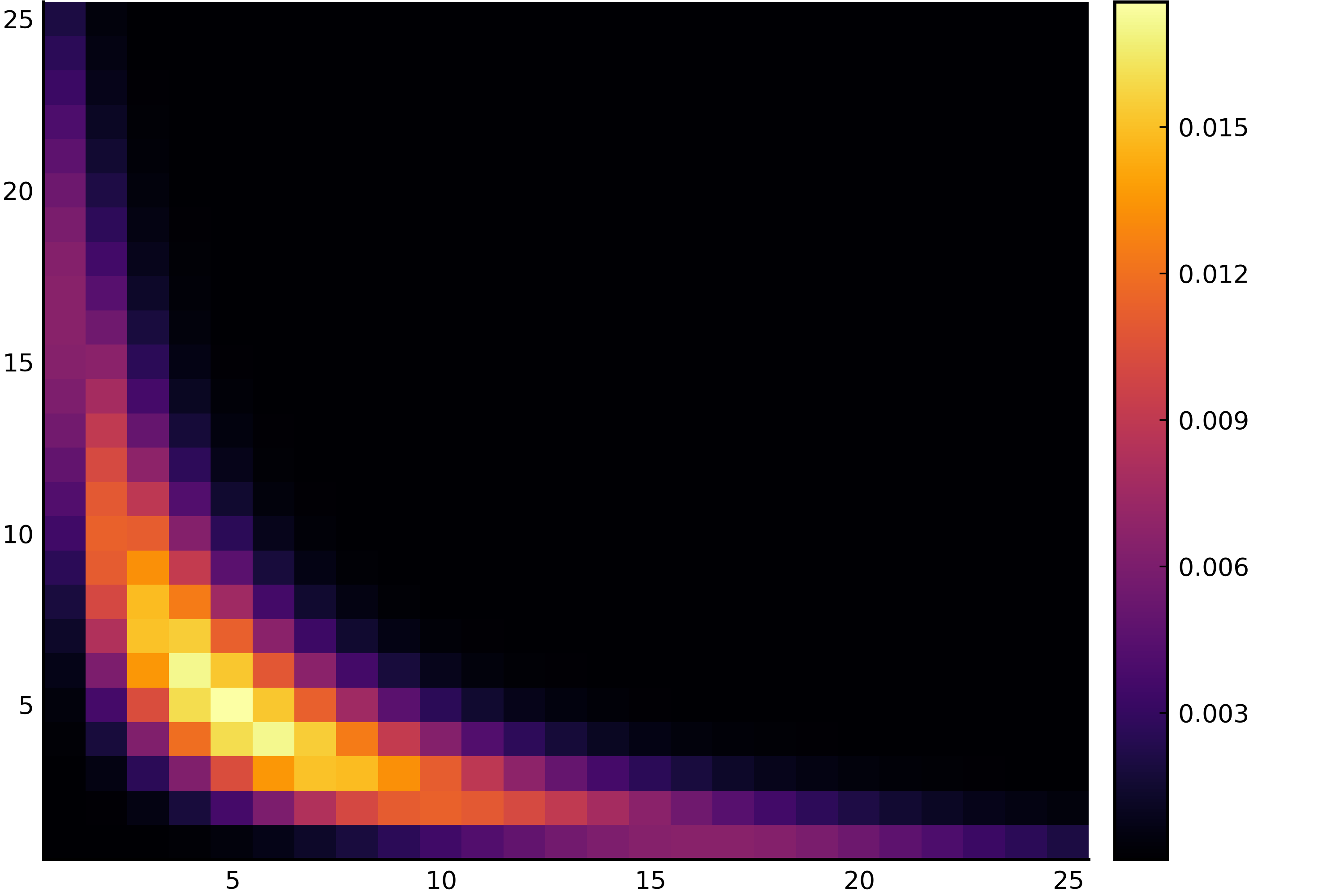}
}
\subfloat[]{\label{sf_TS99_Reference}
\includegraphics[width=60mm]{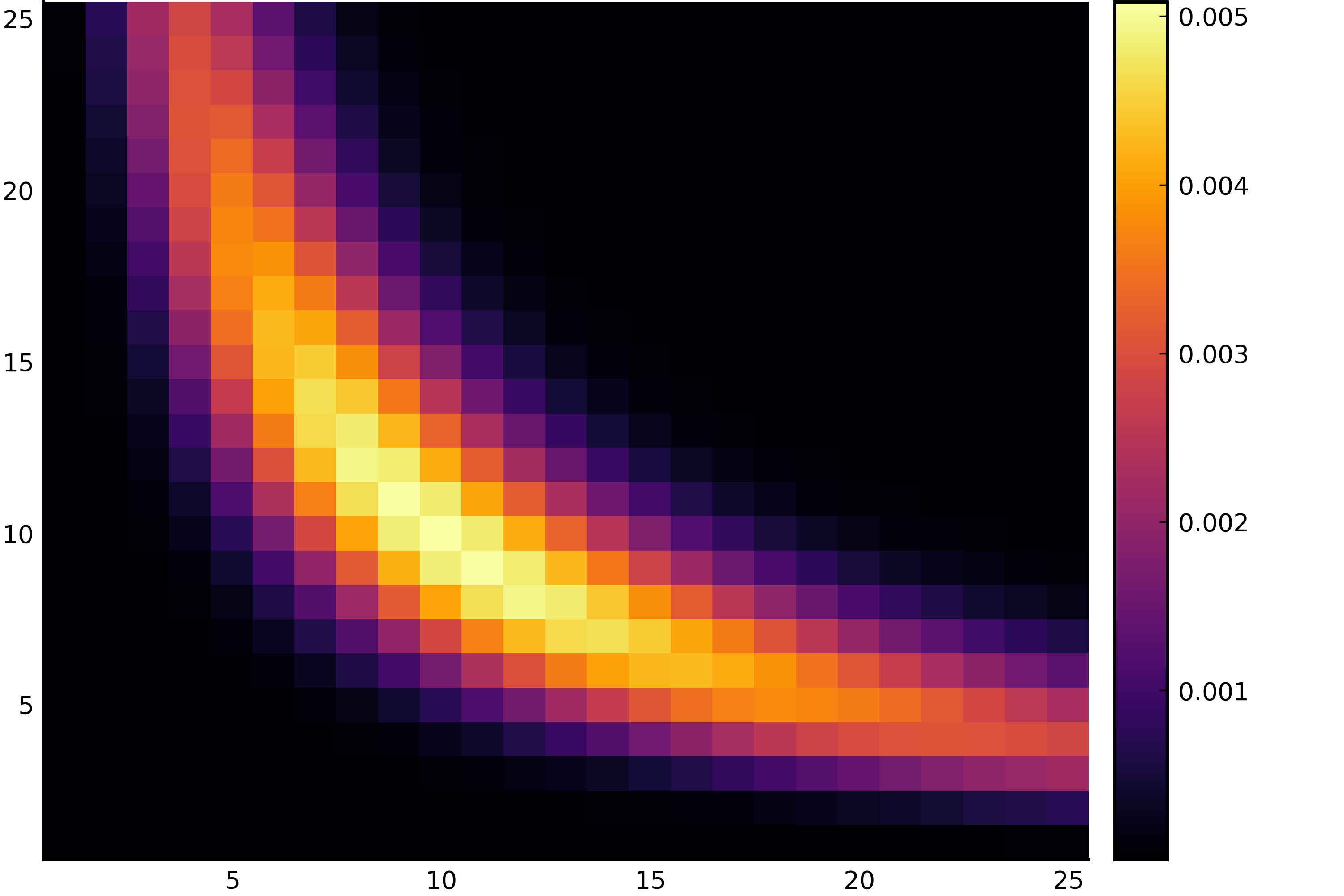}
}\\
\caption{Reference solution for TS(20,1) (a) and TS(90,1) (b).\label{fig_reference_solutions}}
\end{figure}
\end{includefigures}

We are now ready to report our computational findings comparing our algorithms to RTA and LPOA. 

In \cref{fig::GM1,fig::KTS,fig::TS99}, we show (a) the total variation distance of the approximations from the reference solutions, (b) the runtime of the method for computing the approximation, (c) the total variation guarantees for the approximation, (d) the runtime for computing the total variation guarantee and approximation, (e) the relative error gap of the bounds on $\pi r$, and (f) the runtimes for producing those bounds for the $G/M/1$ queue, TS(20,1), and TS(90,1). (Given an upper bound $u$ and a lower bound $\ell$ on the quantity $\pi r$, the relative error gap is defined to be $(u-l)/\pi r$.) For (b), we include the runtime for computing the total variation bound for the approximation $\pi_2^*$, since these experiments are meant to compare algorithms with similar error bound guarantees.

As indicated in the figures, our methods appear to outperform RTA and LPOA both in the quality of approximation and bounds and in runtime for TS(20,1) and TS(90,1) (except for very small truncation sizes). For the $G/M/1$ queue, RTA provides a slightly better approximation at a significant computational cost. Conversely, LPOA has the shortest runtimes but provides an approximation with an almost maximal total variation error (note the black bars corresponding to LPOA's approximation's total variation errors are close to 1 and almost not discernible in \cref{sf_gm1_apxError}). Note that, for the $G/M/1$ queue, the in-boundary is the entire truncation set $A$, which explains RTA's significant computational cost. It also explains the low quality approximation provided by LPOA: the only constraints in the LP of LPOA are the moment bound and bounds on the total probability of the approximating measure (in particular, the first constraint in (4.48) of \citet{kuntzStationaryDistributionsContinuoustime2021} is not in force on any state). That $A=\mathcal{B}_i(A)$ on this model may also explain LPOA's short runtimes: it may be that LPOA's LP's are so simple that the number of decision variables is no longer a reasonable proxy for the difficulty of the LP. The fact that LPOA and RTA have similar relative error gaps despite the fact that LPOA is bound to perform poorly on this example suggests that the dominant bottleneck for these methods on this problem is the moment bound. This may be true more generally on these examples.

It is important to keep in mind that we have not optimized over choices of moment bounds, choices of Lyapunov functions, or choices of Lyapunov functions for moment bounds, for these experiments. We also have not included any of the runtimes for pre-computation of the Lyapunov functions and moment bounds in these experiments. (The runtime for the pre-computation of the Lyapunov functions and moment bounds based on Lyapunov functions are typically insignificant relative to the algorithms we consider on the models and truncation sizes we consider, except for the computation of the moment bounds for the $G/M/1$ queue and TS-20-1. For TS-20-1, we did not run the semi-definite program that produced the moment bound in \citet{kuntzStationaryDistributionsContinuoustime2021}, and so we cannot comment on its runtime relative to the algorithms we consider.)

The code for these experiments is available at \url{https://github.com/alexinfanger/TruncationAlgorithmsforMarkovChains}.

\begin{includefigures}
\begin{figure}[H]
\subfloat[]{\label{sf_gm1_apxError}
\includegraphics[width=60mm]{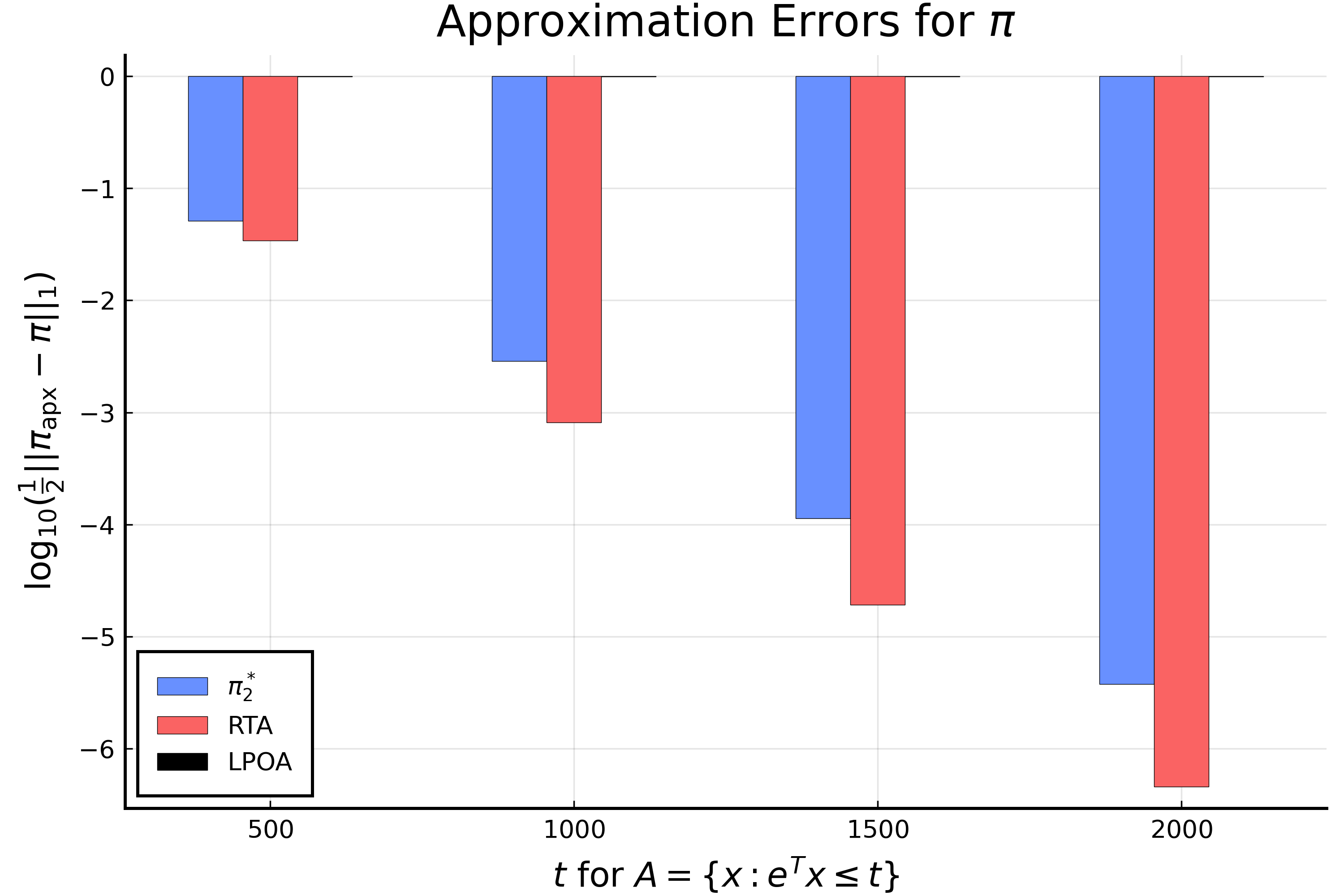}
}
\subfloat[]{\label{sf_gm1_apx_times}
\includegraphics[width=60mm]{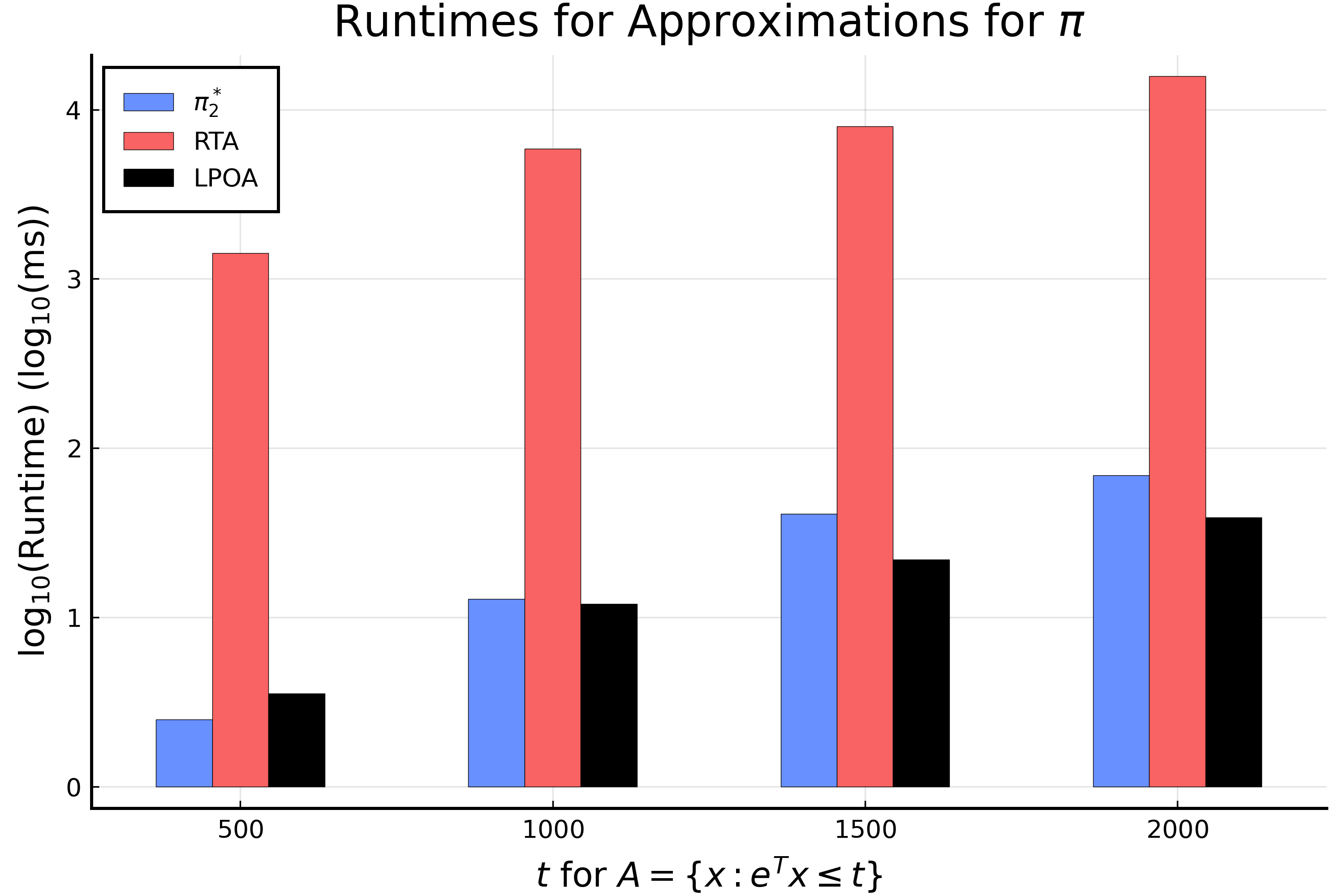}
}\\
\subfloat[]{\label{sf_gm1_tvbounds}
\includegraphics[width=60mm]{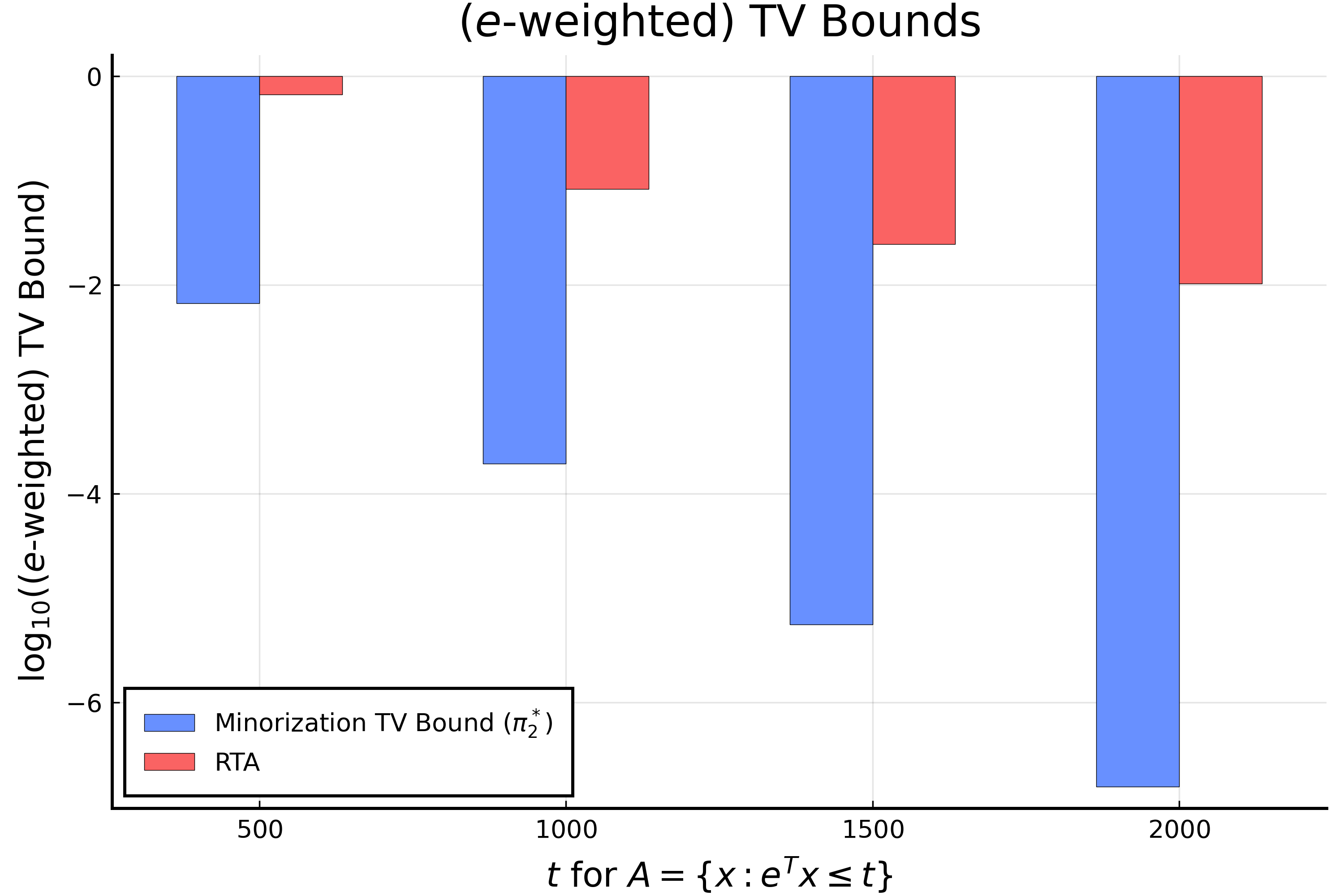}
}
\subfloat[]{\label{sf_gm1_tvbounds_times}
\includegraphics[width=60mm]{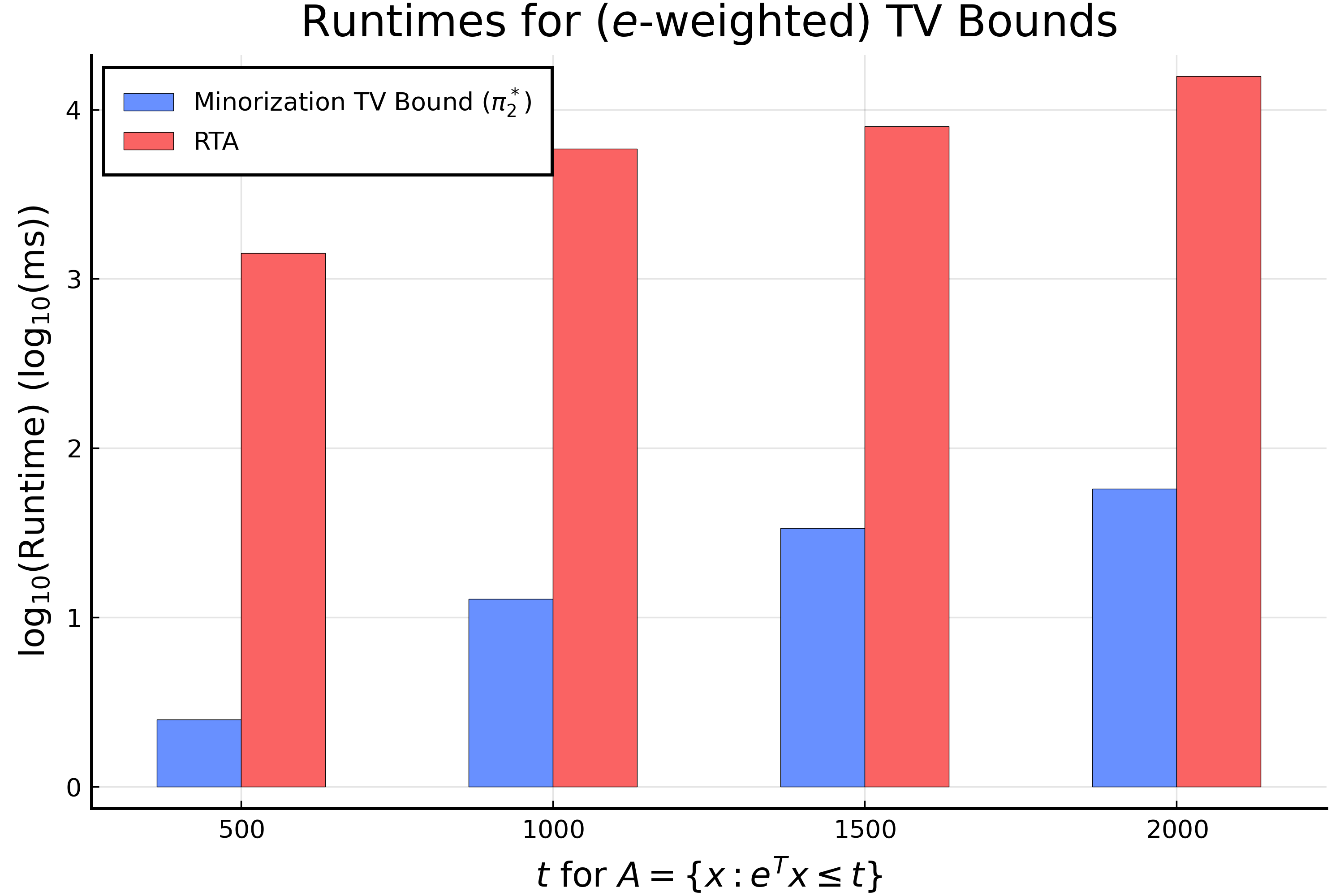}
}\\
\subfloat[]{\label{sf_gm1_convergence_plot}
\includegraphics[width=60mm]{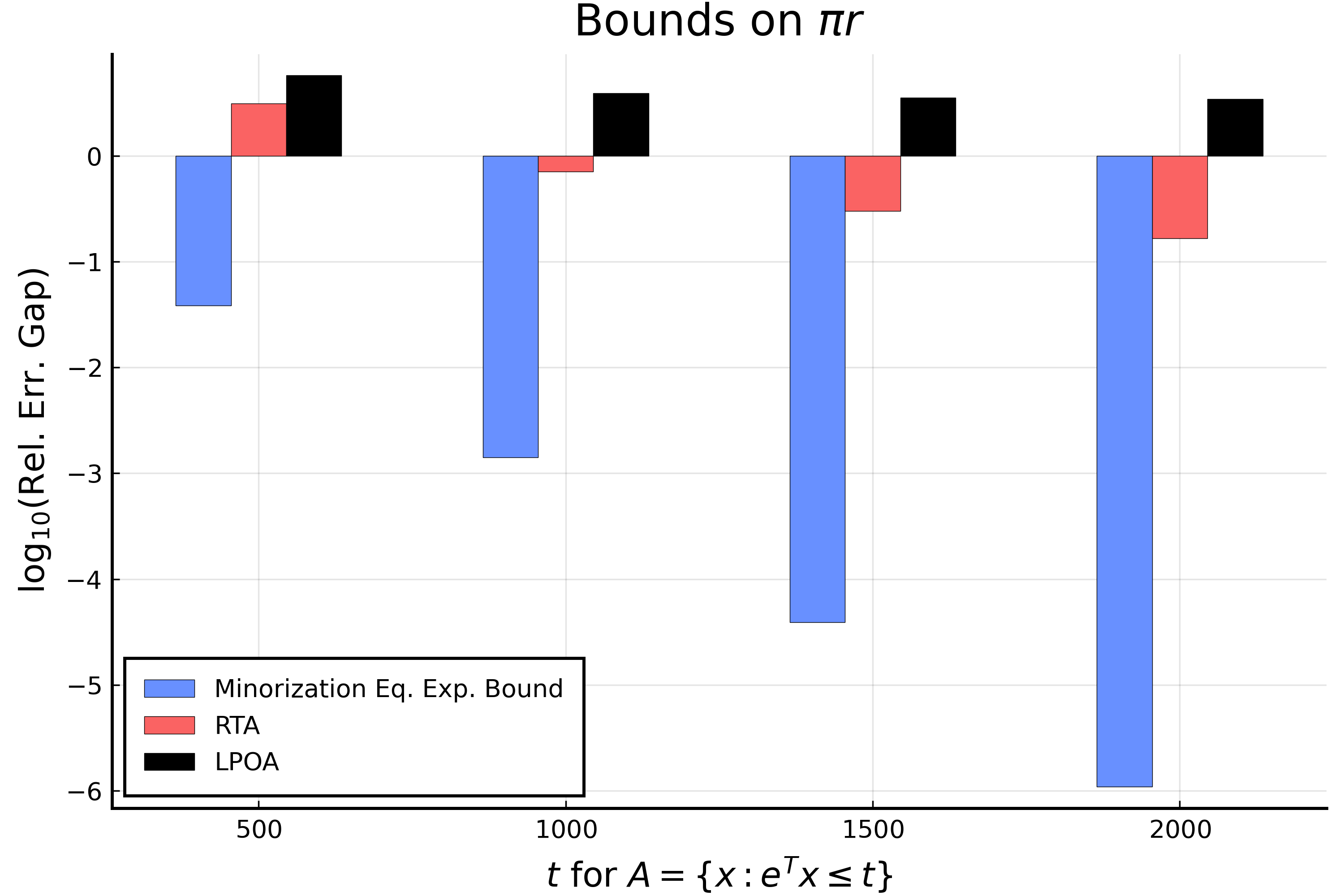}
}
\subfloat[]{\label{sf_gm1_convergence_plot_times}
\includegraphics[width=60mm]{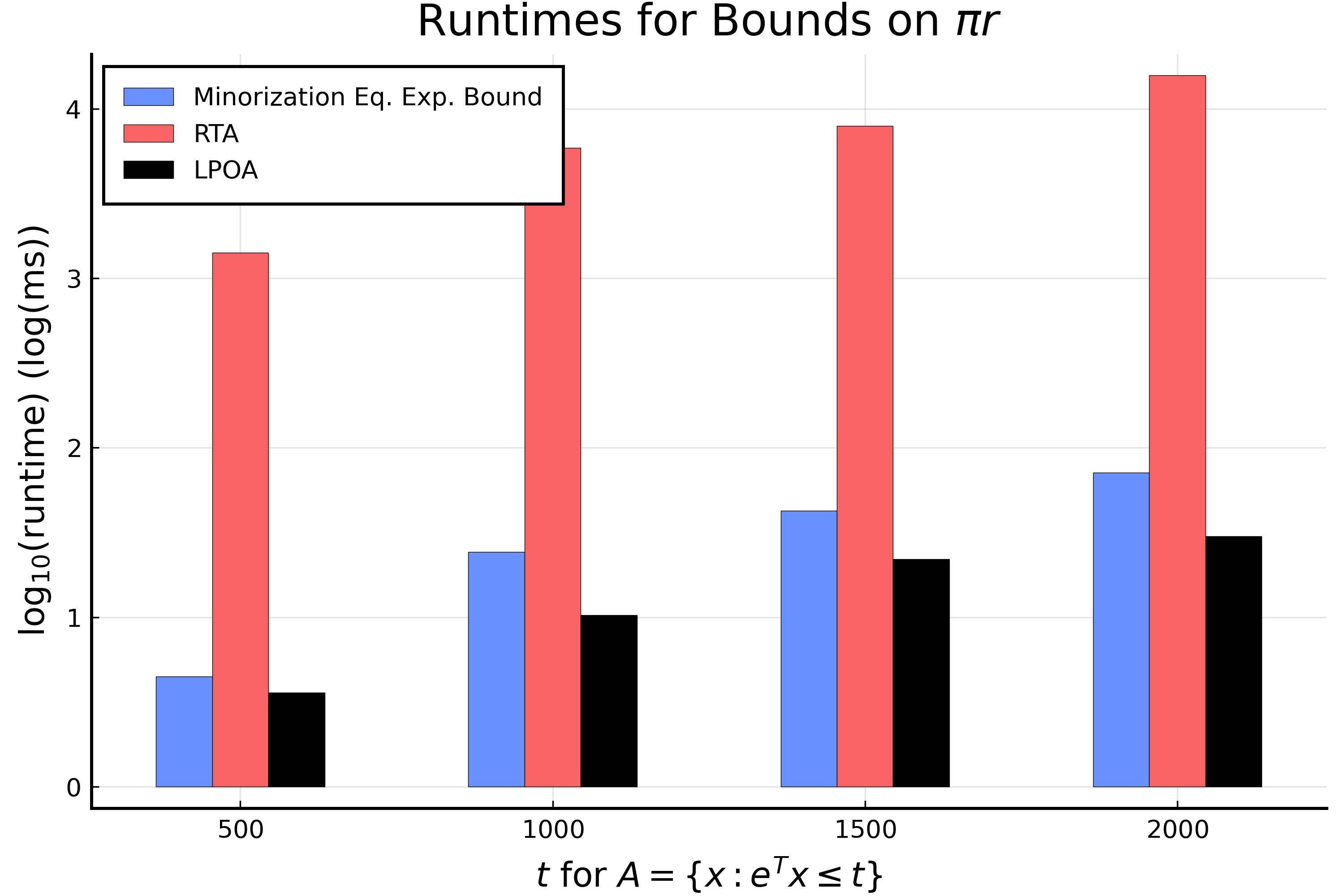}
}
\caption{$G/M/1$ Queue: Comparisons with Other Methods.}\label{fig::GM1}
\end{figure}

\begin{figure}[H]
\subfloat[]{\label{sf_KuntzToggleSwitch_apxError}
\includegraphics[width=60mm]{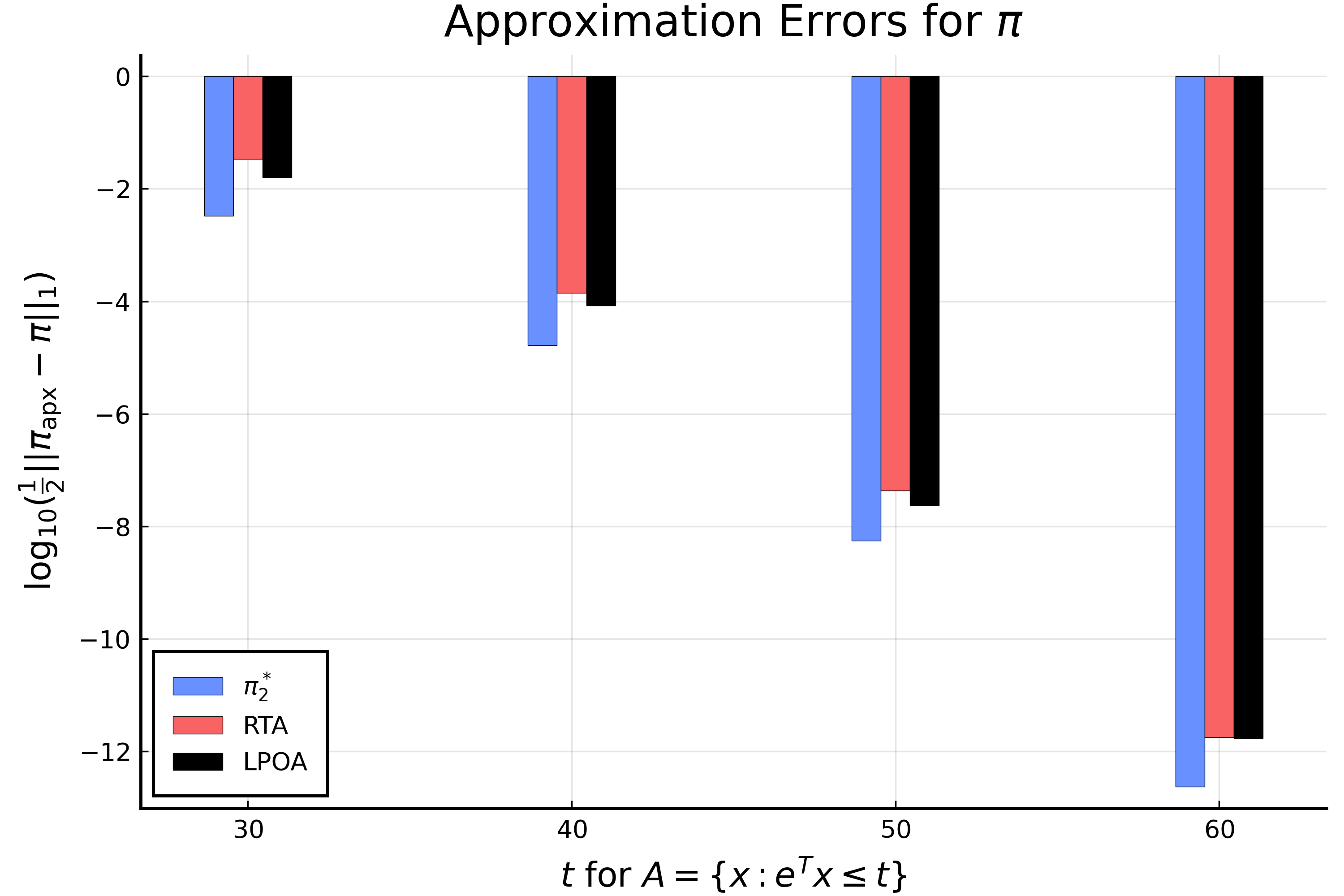}
}
\subfloat[]{\label{sf_KuntzToggleSwitch_apx_times}
\includegraphics[width=60mm]{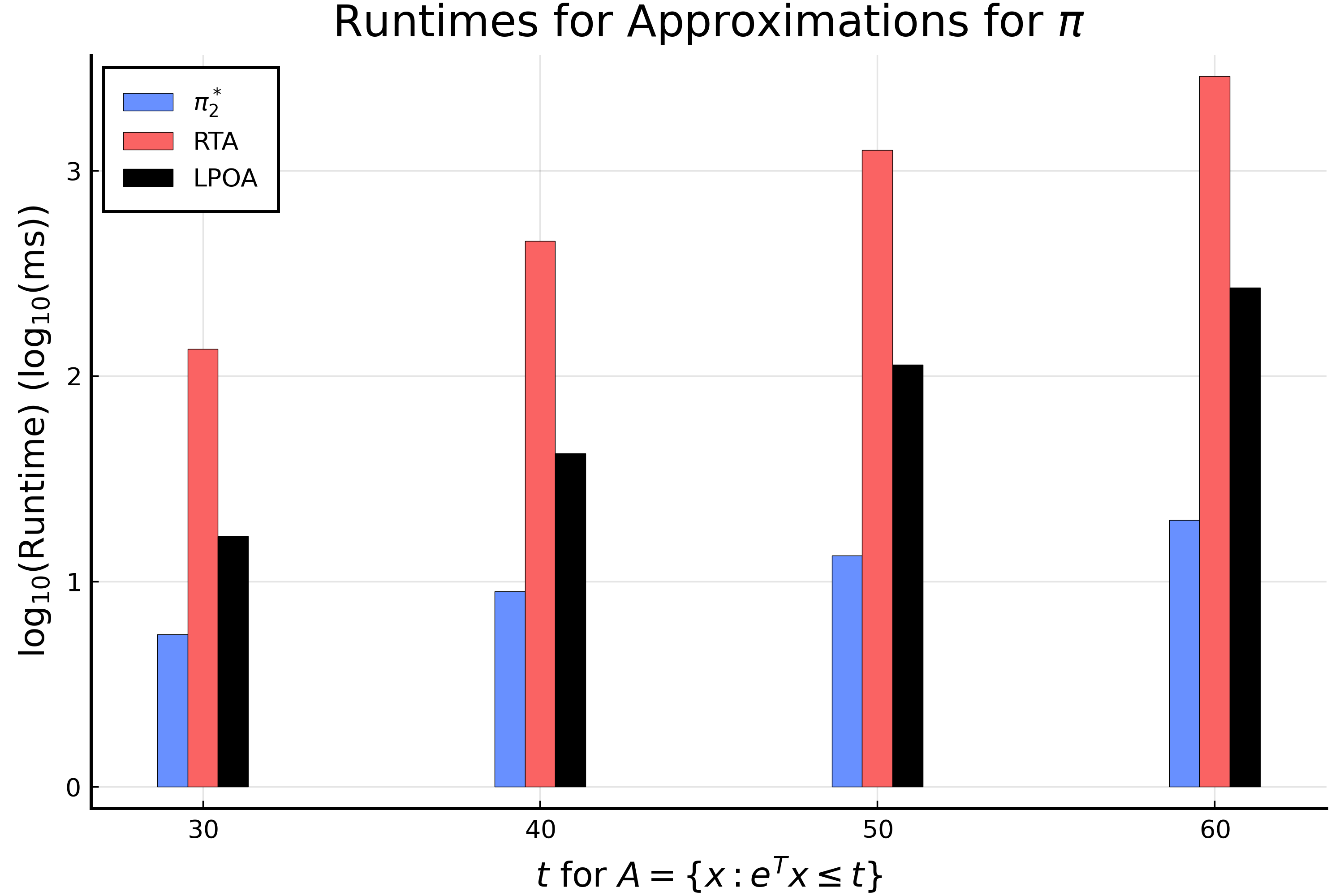}
}\\
\subfloat[]{\label{sf_KuntzToggleSwitch_tvbounds}
\includegraphics[width=60mm]{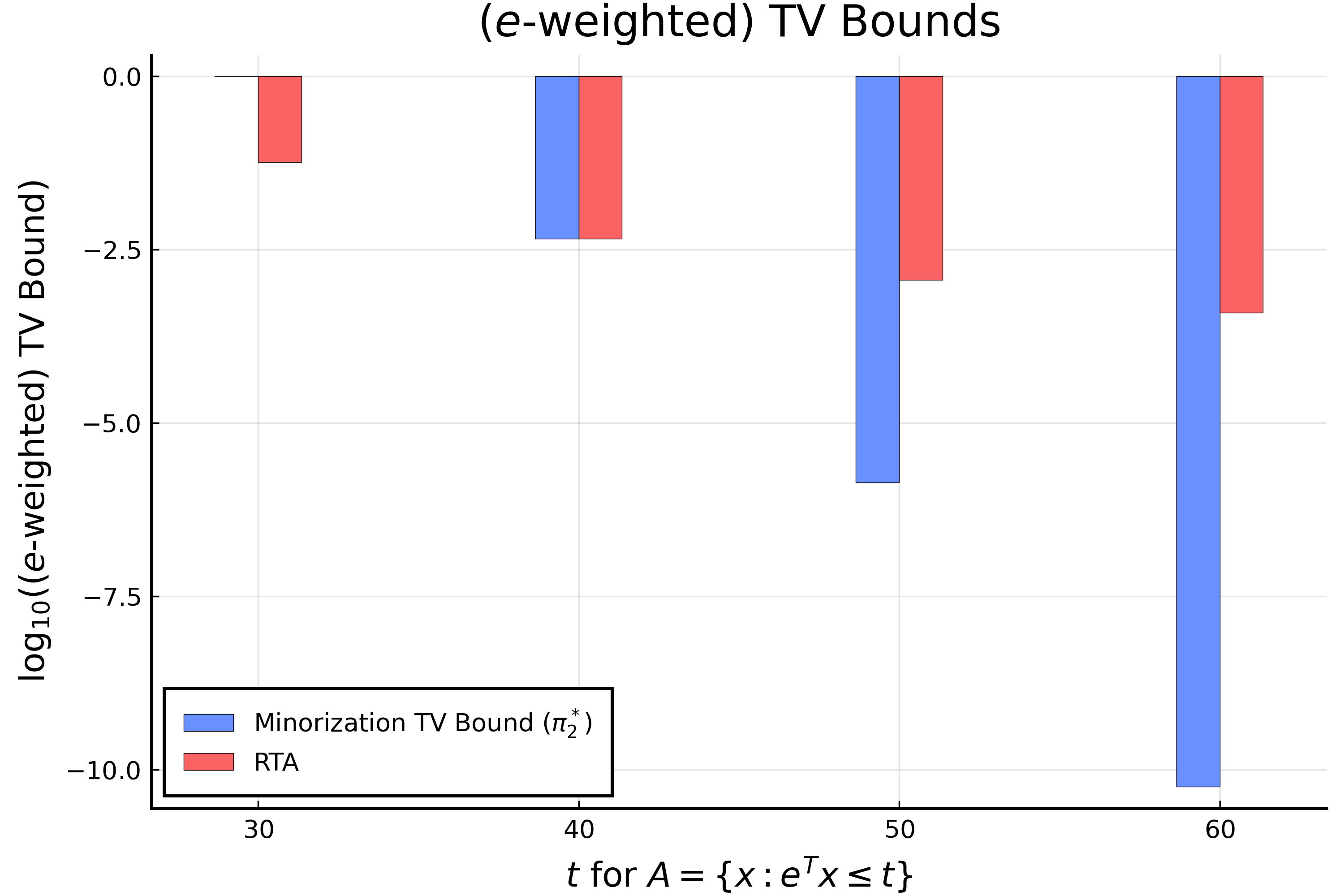}
}
\subfloat[]{\label{sf_KuntzToggleSwitch_tvbounds_times}
\includegraphics[width=60mm]{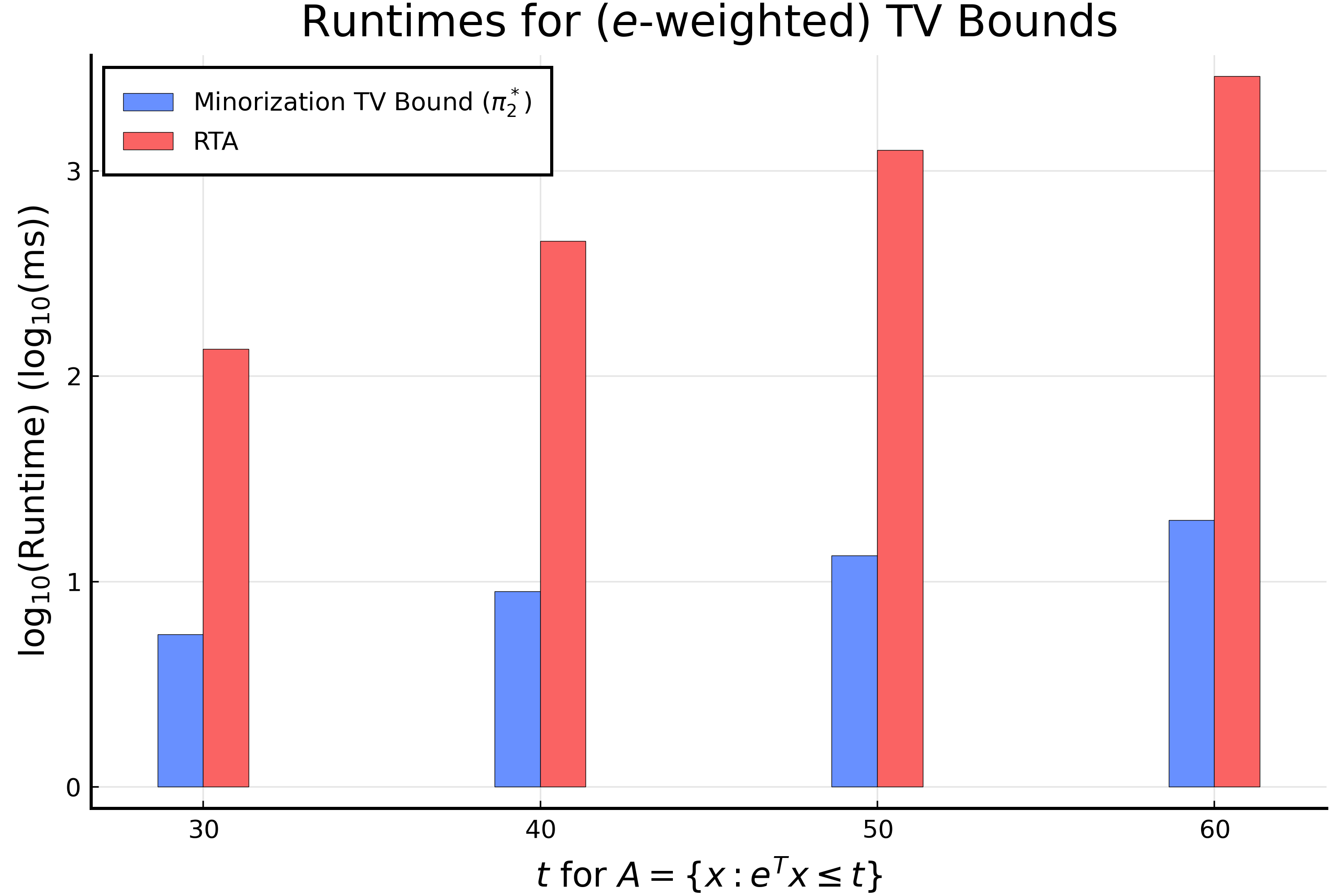}
}\\
\subfloat[]{\label{sf_KuntzToggleSwitch_convergence_plot}
\includegraphics[width=60mm]{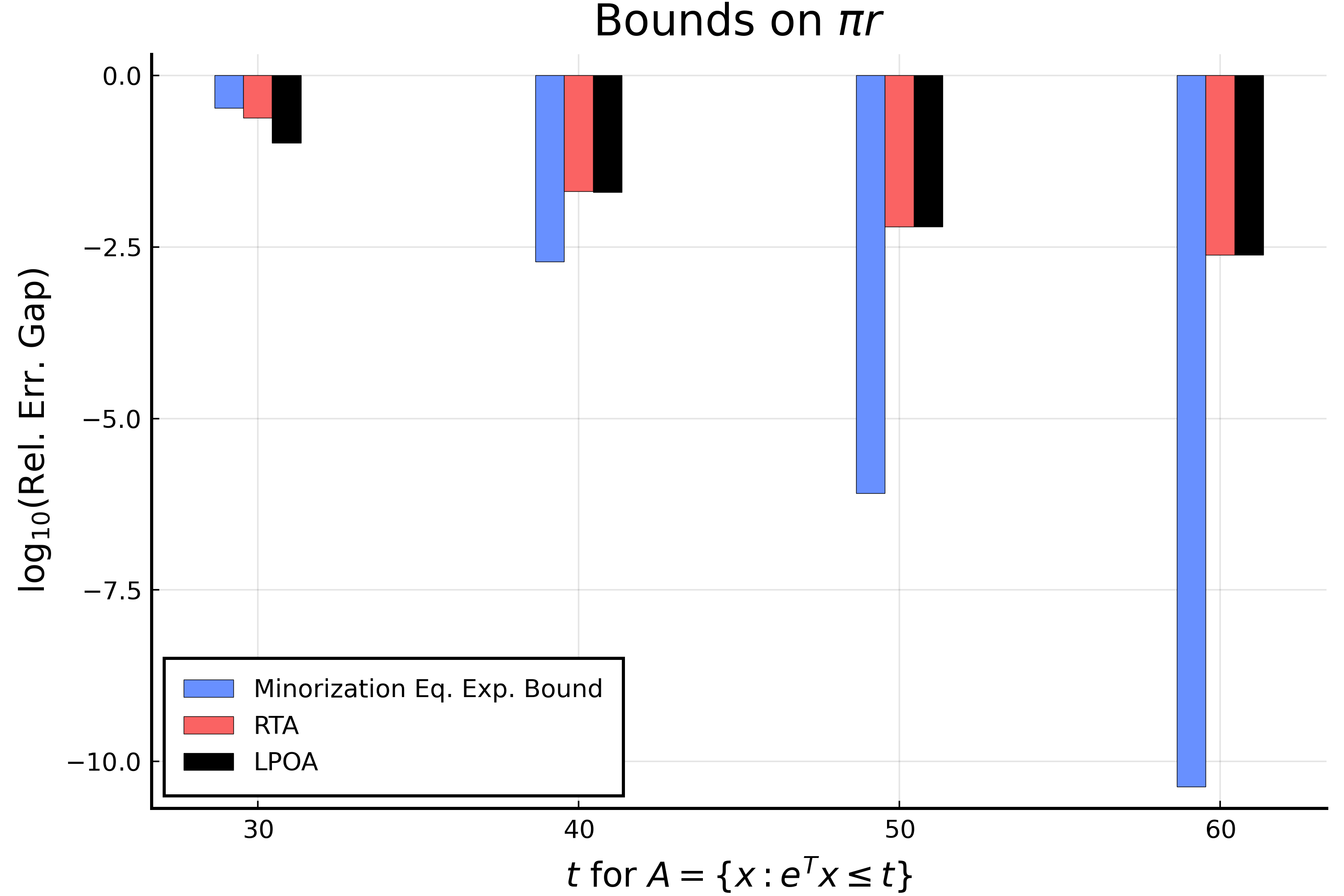}
}
\subfloat[]{\label{sf_KuntzToggleSwitch_convergence_plot_times}
\includegraphics[width=60mm]{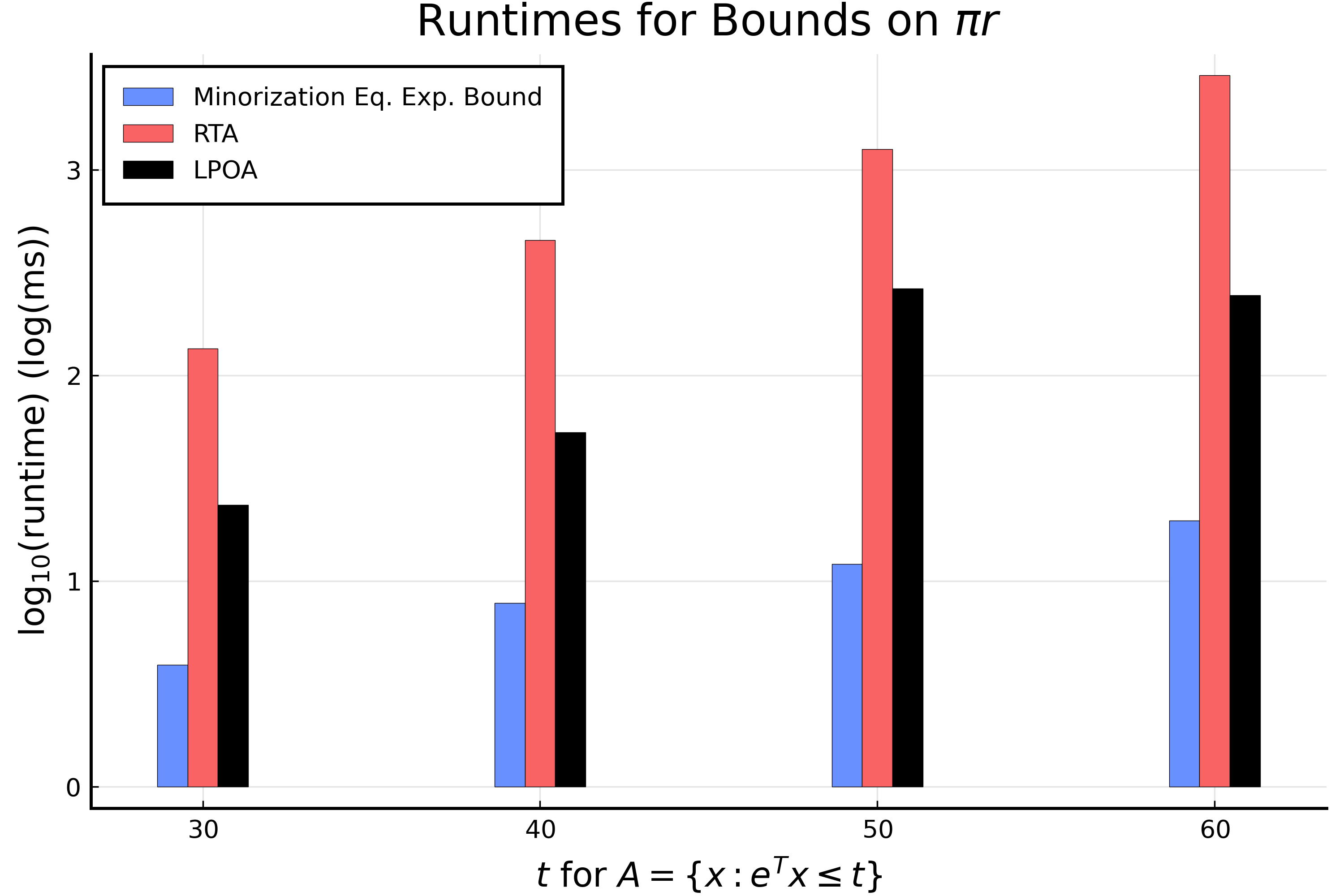}
}
\caption{TS(20,1): Comparisons with Other Methods.\label{fig::KTS}}
\end{figure}

\begin{figure}[H]
\subfloat[]{\label{sf_ToggleSwitch99_apxError}
\includegraphics[width=60mm]{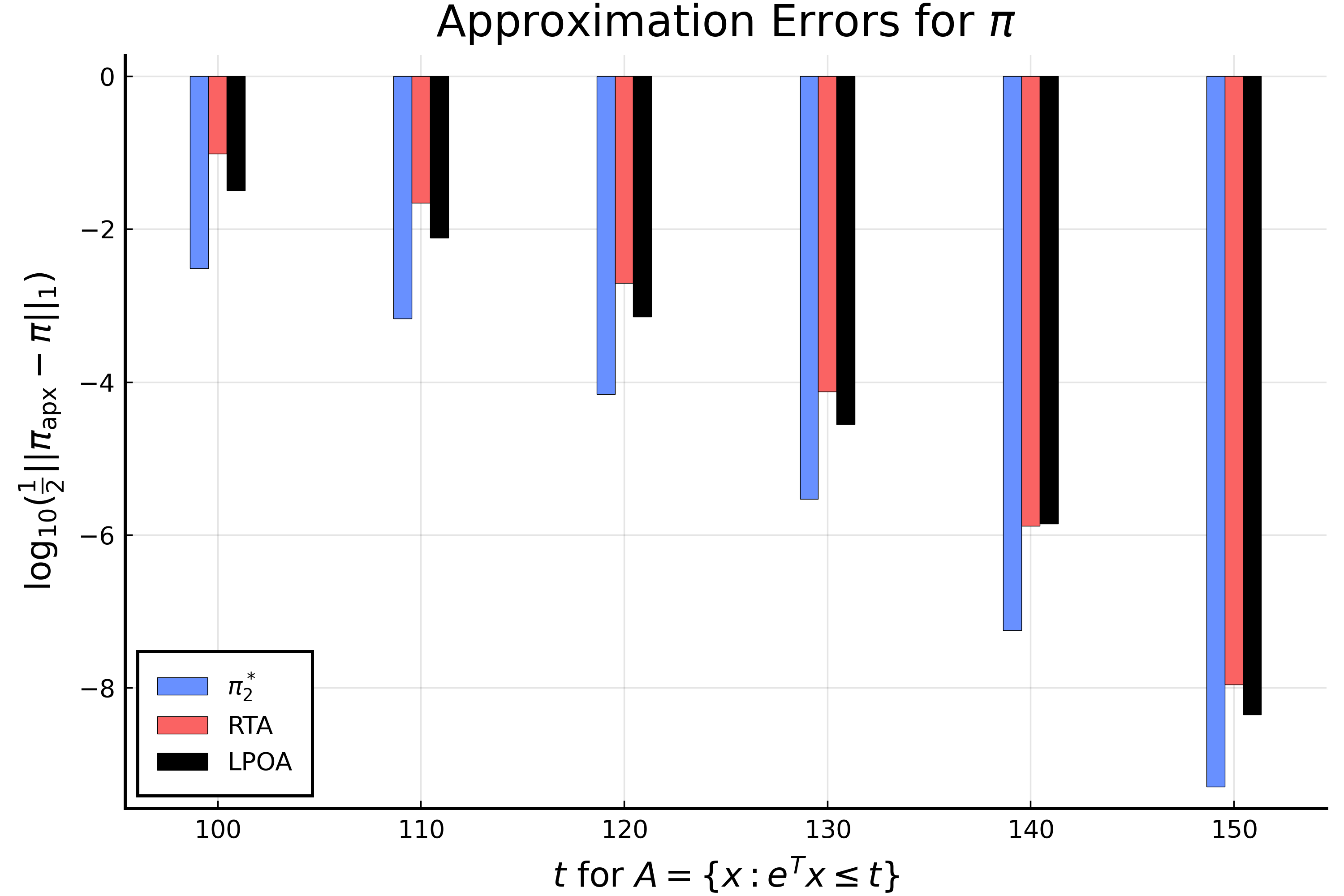}
}
\subfloat[]{\label{sf_ToggleSwitch99_apx_times}
\includegraphics[width=60mm]{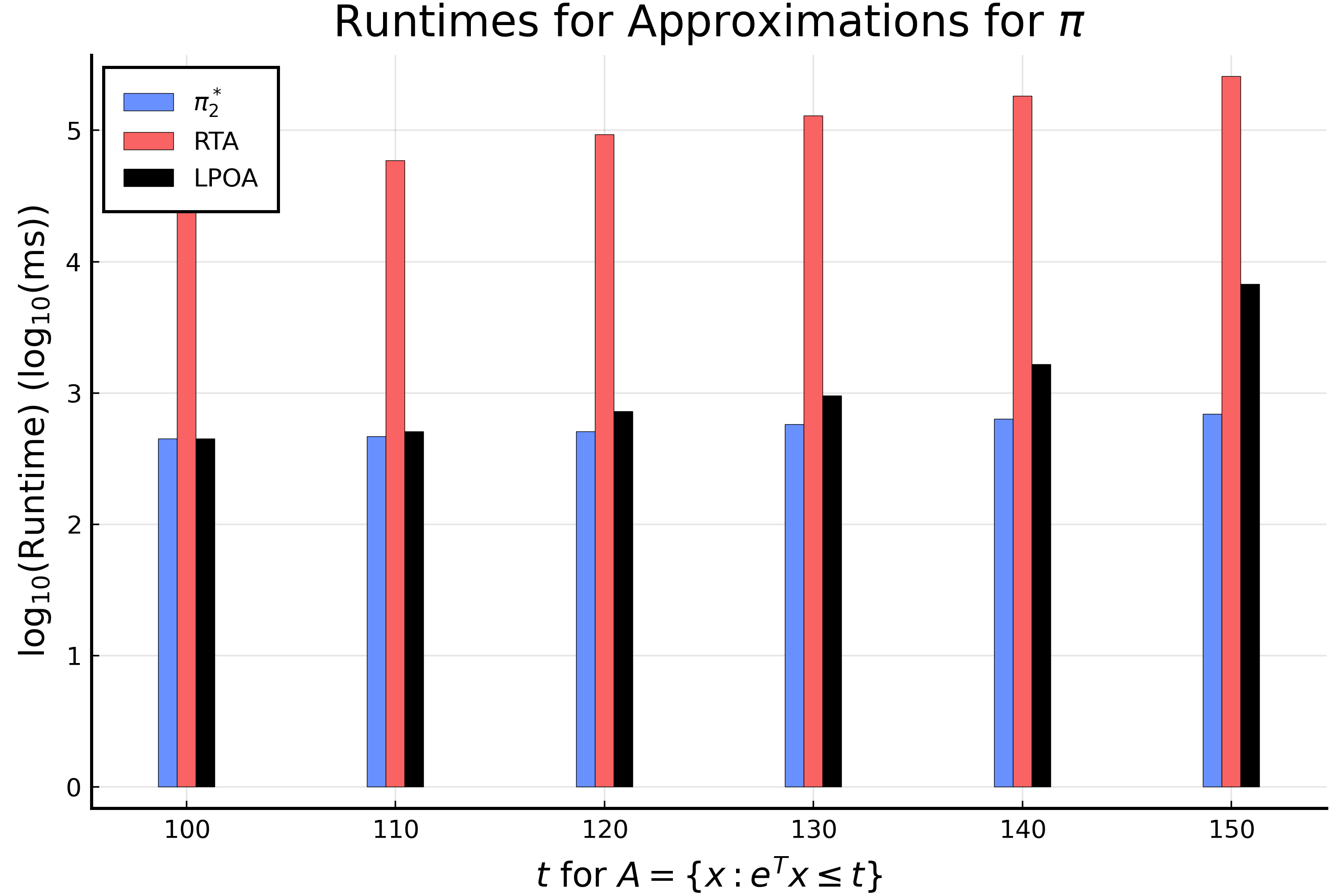}
}\\
\subfloat[]{\label{sf_ToggleSwitch99_tvbounds}
\includegraphics[width=60mm]{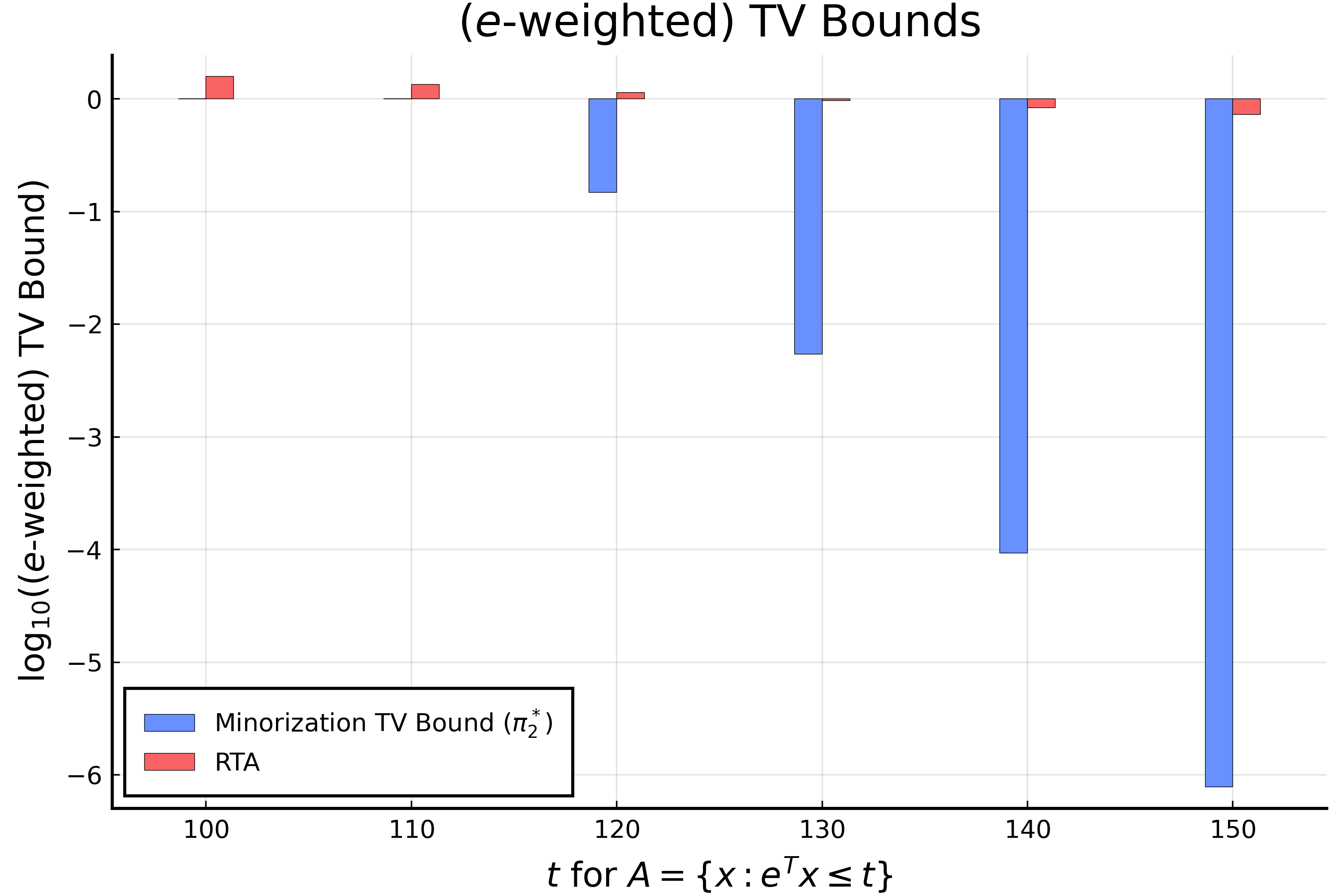}
}
\subfloat[]{\label{sf_ToggleSwitch99_tvbounds_times}
\includegraphics[width=60mm]{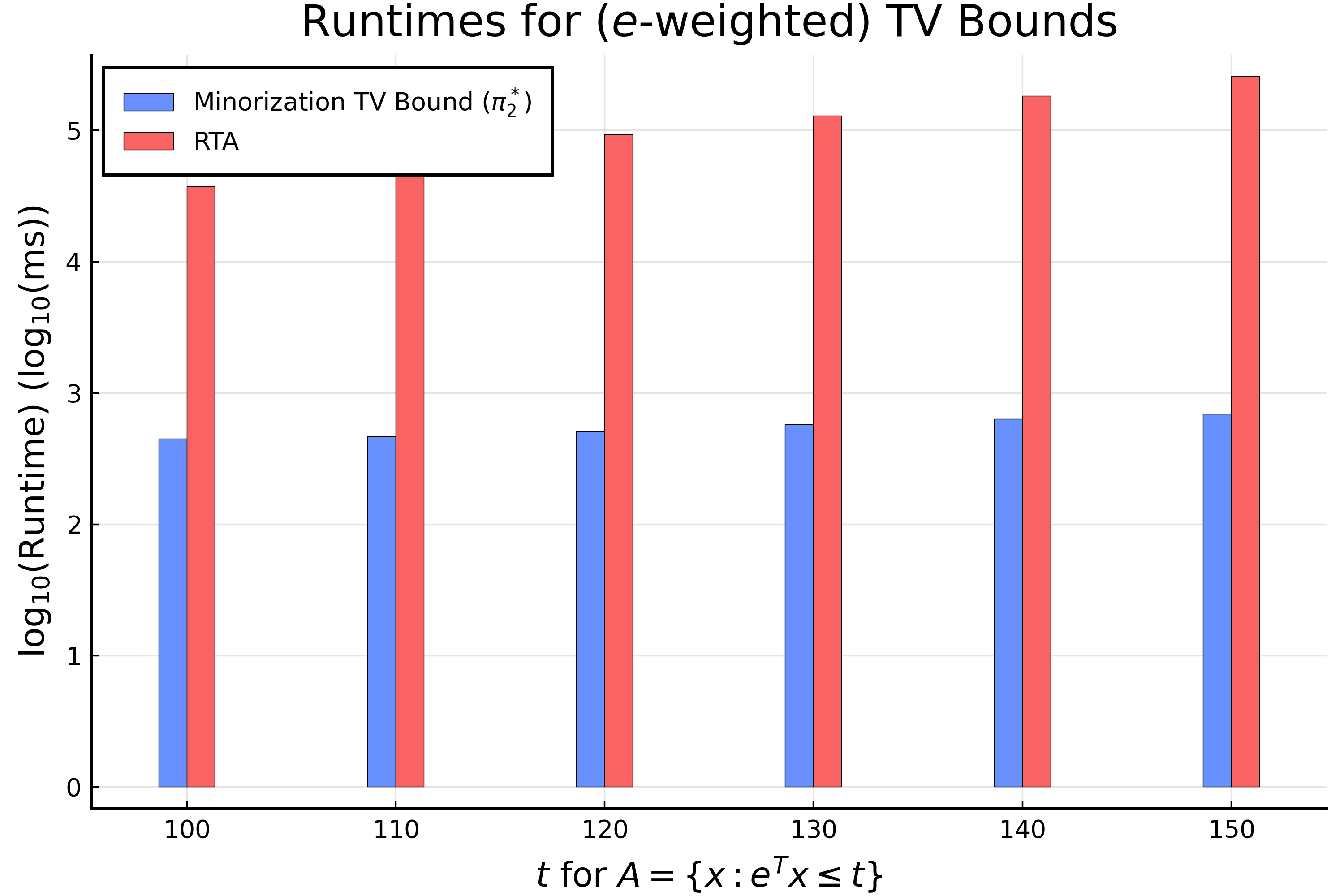}
}\\
\subfloat[]{\label{sf_ToggleSwitch99_convergence_plot}
\includegraphics[width=60mm]{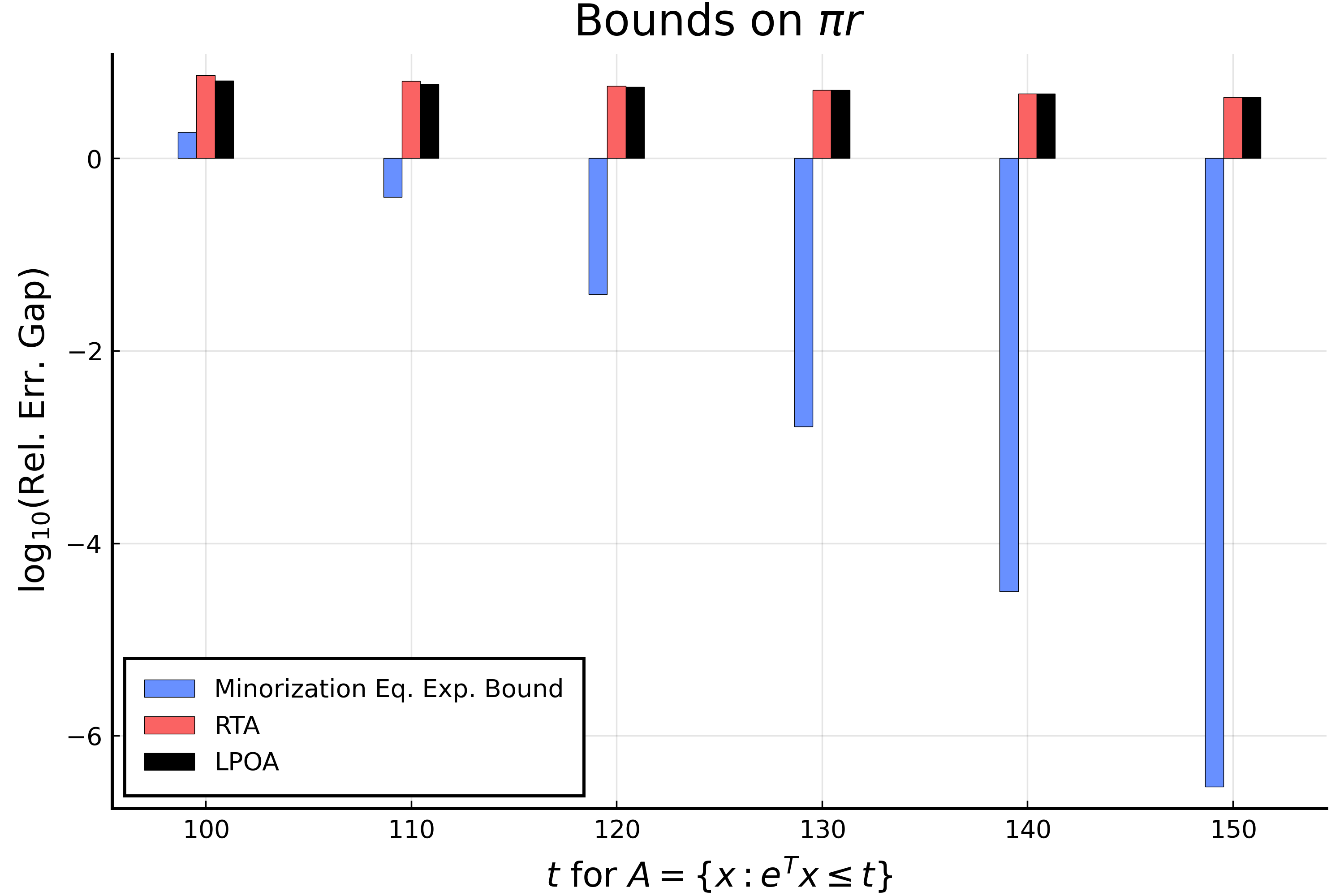}
}
\subfloat[]{\label{sf_ToggleSwitch99_convergence_plot_times}
\includegraphics[width=60mm]{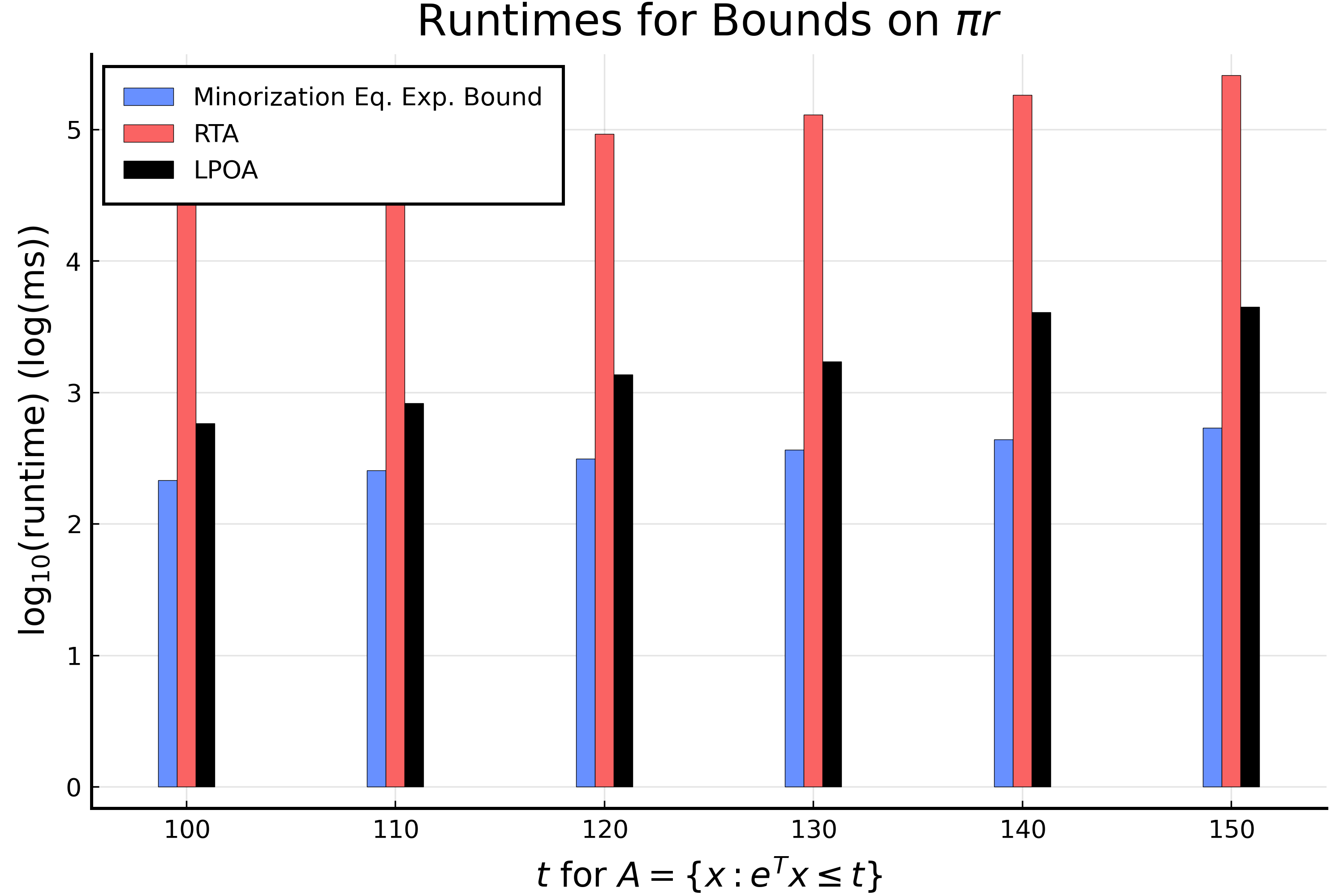}
}\caption{TS(90,1): Comparisons with Other Methods.\label{fig::TS99}}
\end{figure}
\end{includefigures}

\bibliographystyle{apalike}
\bibliography{TruncationPaper}

\end{document}